\documentclass[final,onefignum,onetabnum]{siamart171218}

\usepackage{amsmath}
\usepackage{amsfonts}
\usepackage{amssymb}
\usepackage{graphicx,subcaption}
\usepackage[skip=2pt,font=footnotesize,labelfont={bf,it},textfont=it]{caption}
\usepackage{enumitem}

\newcommand {\eps}{\varepsilon}
\newcommand {\im}{\mathrm{i}}
\newcommand {\N}{ \mathbb{N} }
\newcommand {\C}{ \mathbb{C} }
\newcommand {\R}{ \mathbb{R} }
\newcommand {\calO} {{\cal O}}

\DeclareMathOperator*{\argmax}{arg\,max}
\DeclareMathOperator*{\Ai}{\rm Ai}

\newtheorem{ass}[theorem]{Assumption}
\newtheorem{rem}[theorem]{Remark}
\newtheorem{exam}[theorem]{Example}

\newcommand {\mtrx}[2] {\left(\begin{array}{#1} #2 \end{array} \right)}
\newcommand\todo[1]{\textcolor{black}{#1}}

\author{
Anton Arnold \thanks{Institute of Analysis and Scientific Computing,
Wiedner Hauptstr. 8-10, 1040 Wien (\texttt{anton.arnold@tuwien.ac.at}).
} 
\and
Kirian D\"opfner \thanks{Institute of Analysis and Scientific Computing, 
Wiedner Hauptstr. 8-10, 1040 Wien (\texttt{kirian.doepfner@tuwien.ac.at}).
}
}

\title{Stationary Schr\"odinger equation in the semi-classical limit: WKB-based scheme coupled to a turning point}

\begin{document}
\maketitle


\begin{abstract} 
This paper is concerned with the efficient numerical treatment of 1D stationary Schr\"odinger equations in the semi-classical limit when including a turning point of first order. As such it is an extension of the paper \cite{AN18}, where turning points still had to be excluded. For the considered scattering problems we show that the wave function asymptotically blows up at the turning point as the scaled Planck constant $\eps\to0$, which is a key challenge for the analysis. Assuming that the given potential is linear or quadratic in a small neighborhood of the turning point, the problem is analytically solvable on that subinterval in terms of Airy or parabolic cylinder functions, respectively. Away from the turning point, the analytical solution is coupled to a numerical solution that is based on a WKB-marching method -- using a coarse grid even for highly oscillatory solutions. We provide an error analysis for the hybrid analytic-numerical problem up to the turning point (where the solution is 
asymptotically unbounded) and illustrate it in numerical experiments: If the phase of the problem is explicitly computable, the hybrid scheme is asymptotically correct w.r.t.\ $\eps$. If the phase is obtained with a quadrature rule of, e.g., order 4, then the spatial grid size has the limitation $h=\calO(\eps^{7/12})$ which is slightly worse than the $h=\calO(\eps^{1/2})$ restriction in the case without a turning point.
\end{abstract}






\section{Introduction}\label{sec:intro}

This paper is concerned with the numerical treatment of the stationary, one-dimensional Schr\"odinger equation \begin{equation} \label{eqn:1d_schroed}
\eps^2\psi^{\prime\prime}(x) + a(x) \psi(x) = 0\;, \quad x\in\R
\end{equation}
in the semi-classical limit $\eps\to0$.
Here, $0<\eps\ll1$ is the rescaled Planck constant ($\eps := \frac{\hbar}{\sqrt{2m}} $), 
$\psi(x)$ the (possibly complex valued) Schr\"odinger wavefunction, and the (given) real valued coefficient function $a(x)$ is related to the potential. For $a(x)>0$ and $\eps$ ``small'', the solution is highly oscillatory, and efficient numerical schemes have been developed, e.g., in \cite{ABN11, LJL05}.
The novel feature of this work is to include one \emph{turning point\/} of first order at $x=0$ (i.e.\ $a(0)=0$, $a'(0)>0$). Then, $x<0$ represents the evanescent region where the solution $\psi$ decreases exponentially, possibly including a pronounced boundary layer. $x>0$ is the oscillatory region, where the solution exhibits rapid oscillations with (local) wave length $\lambda(x)=\tfrac{2\pi\eps}{\sqrt{a(x)}}$. Hence, the situation at hand is a classical multi-scale problem. 

Standard numerical methods (e.g., \cite{IB95, IB97}) for \eqref{eqn:1d_schroed} --particularly in the highly oscillatory regime-- are costly and inefficient as they would require to resolve the oscillations by choosing a spatial grid with step size $h=\calO(\eps)$. By contrast, we are aiming here at a numerical method on a coarse grid with $h>\lambda$, while still recovering the fine structures of the solution. Our strategy is built upon the following works: For the purely oscillatory case $a(x)\ge\tau_1>0$ in the semi-classical regime, WKB-based marching methods (named after the physicists Wentzel, Kramers, Brillouin) were developed in \cite{ABN11, JL03, LJL05}. They allow to reduce the grid limitation to at least $h=\calO(\sqrt\eps)$. 
For the evanescent case $a(x)\le\tau_3<0$ (as $\eps\to0$), a WKB-based multi-scale FEM was introduced in \S3 of \cite{Neg05}.
A hybrid method to couple both of these regimes was recently introduced and analyzed in \cite{AN18}; it consists of a (non-overlapping) domain decomposition method. 
Turning points were excluded there, since (to the authors' knowledge) no $\eps$--uniformly accurate method for \eqref{eqn:1d_schroed} including turning points has been developed so far. Hence, the function $a$ in \cite{AN18} was assumed to have a jump discontinuity at the interface between the evanescent and the oscillatory regimes. In this work we shall extend the setting by including turning points of special form. Rather than providing a numerical scheme that can handle general turning points (which is unknown so far), this paper is more a feasibility study to identify the involved problems. 
One of the key challenges towards a uniformly accurate scheme for turning points is the fact that the continuous solutions $\psi_\eps$ asymptotically blow up (as $\eps\to0$) at the turning point (for a scattering problem to be specified in Section \ref{sec:scat_mod} below).

Hence, we shall not attempt here to make the WKB-marching method from \cite{ABN11} extendable up to the turning point. As a first step towards a full semi-classical method including turning points, we shall rather assume that $a$ is either a linear or quadratic function of $x$ close to the turning point, say on $[0,x_1]$. On that interval, the solution of \eqref{eqn:1d_schroed} is then an Airy function or, respectively, a parabolic cylinder function. This will lead to a hybrid method that is analytic on $[0,x_1]$ and numerical for $x>x_1$. 
Our strategy thus combines the WKB-method from \cite{ABN11} (away from the turning point) with the philosophy of \cite{Hal13}, \S15.5 (i.e.\ a linear approximation 
 of the potential in the first cell adjacent to the turning point). 
Nevertheless, this coupled problem still includes the effects of the turning point and the problems with handling it: The solution to \eqref{eqn:1d_schroed} --as a boundary value problem (BVP)-- becomes unbounded at the turning point in the semi-classical limit. Therefore, as $\eps\to0$, the errors of standard numerical methods would become unbounded there as well (since numerical errors in a BVP are non-local and pollute the whole interval). But for our hybrid method we shall still derive an error estimate \emph{up to the turning point\/}, and this error even decreases with $\eps$. We first note that, although the analytic solution form will be known on $[0,x_1]$, that (asymptotically unbounded) solution part is polluted by the numerical error at the boundaries. 
For our estimates it will be crucial that the WKB-method from \cite{ABN11} is not only \emph{uniformly accurate\/} w.r.t.\ $\varepsilon$ but even \emph{asymptotically correct\/}, i.e.\ the numerical error goes to zero with $\varepsilon\to0$. This will allow to over-compensate the unbounded growth of $|\psi_\varepsilon(0)|$ as $\varepsilon\to0$. 

This paper is organized as follows: In Section \ref{sec:scat_mod} we specify the scattering problem to be discussed, and in Section \ref{sec:ana_prob} we rewrite the scattering-BVP as an initial value problem (IVP), coupled to an Airy function solution close to the turning point (for the case of a linear potential on $[0,x_1]$). Section \ref{sec:blow_up} illustrates the blow-up of $\psi_\eps$ in the semi-classical limit.  Section \ref{sec:numeth_errana} is the core part of this work: We extend the WKB-error analysis from \cite{ABN11} to the hybrid problem at hand. This follows the strategy for the domain decomposition method in \cite{AN18}, but is more subtle here -- due to the unboundedness of $\psi_\eps$. A numerical illustration of the proved error estimates closes that section. Section \ref{sec:mod3} extends the previous analysis to the case of quadratic potentials -- close to the turning point.
In the final Section \ref{sec7} we briefly discuss possible extensions to more general potentials close to the turning point.


\section{Scattering Model}\label{sec:scat_mod}

Highly oscillatory problems similar to \eqref{eqn:1d_schroed} appear in a wide range of applications (e.g., electromagnetic and acoustic scattering, quantum physics). Our interest in this problem is motivated by the electron transport in nano-scale semiconductor devices, which will determine the details of our set-up. 1D models are of course idealizations, but quite appropriate, e.g., for resonant tunneling diodes \cite{BP06}. 
In this application $\psi$ represents the quantum mechanical wave function. The derived macroscopic quantities of interest to practitioners are the particle density $n(x):=|\psi(x)|^2$ and the current density $j(x):=\eps \Im(\bar\psi(x)\psi'(x))$, see \cite{ABN11} for more details. While $\psi$ is highly oscillatory, $n$ and $j$ are not; but they can only be obtained from the wave function.

We consider the \emph{internal domain\/} $x\in[0,1]$, which corresponds to the semiconductor device. Moreover, we assume that electrons are injected from the right boundary (or lead) with the prescribed energy $E$. The coefficient function $a$ in \eqref{eqn:1d_schroed} is then given by $a(x) := E-V(x)$ where $V(x)$ is the (prescribed) electrostatic potential of the problem. 
In reality, electrons are injected into a device as a statistical mixture with continuous energies $E\ge E_0$ \cite{FG97}. Therefore, given a potential $V$, a whole energy interval will give rise to \emph{turning points\/}\footnote{There, a classical particle with energy $E$ in the potential $V(x)$ would change directions, hence the name.}, i.e.\ zeros of $a$ within the interval $[0,1]$. To simplify the presentation we shall assume that the only turning point is located at $x=0$, and we shall consider only one such injection energy $E$. Specifically, we shall make the following assumptions for the given potential $V$ (see Fig.\ \ref{fig:potential}).

\begin{ass}\label{ass:coeff_a}\leavevmode
\begin{enumerate}[label = \alph*)]
\item{ $V\in \mathcal{C}(\R;\R)$. Potential jumps are excluded here only for simplicity. Without difficulty, they could be included inside the interior domain $(0,1)$ by restarting the IVP (from Section \ref{sec:ana_prob}) at jump points.}\label{item:cont_pot}
\item{ Let $V(x)<V(0)$ for $x\in(0,1]$ (to exclude further turning points besides of $x=0$).}\label{item:pot_1tp}
\item{ The potential in the left exterior domain (i.e. $x<0$)\todo{, and also in a (small) neighborhood of $x=0$,} is linear (for simplicity of the hybrid problem). We also assume (w.l.o.g.) that it has slope $-1$. \todo{More precisely we assume that $\exists\,x_1\in(0,1)$ such that $a(x)=x$ for $x\leq x_1$. Hence, \eqref{eqn:1d_schroed} is the (scaled) Airy equation for $x\leq x_1$. }}\label{item:pot_lext_tp}
\item{ The potential in the right exterior domain (i.e. $x>1$) is constant with value $V(1)$.}\label{item:pot_rext}
\end{enumerate}
\end{ass}

Hence, this scattering problem is oscillatory for $x>0$, evanescent for $x<0$, and it has a turning point of order $1$ at $x=0$. 
Since $V(x)>E$ for $x<0$, the injected wave function is fully reflected. Due to \ref{item:cont_pot} and \ref{item:pot_1tp} $\exists\,\tau_1>0$ such that $\tau_1\leq a(x)$ on $[x_1,1]$, which is an important assumption for the WKB-marching method from \cite{ABN11}.

In a realistic device model, it would of course be appropriate to assume that the potential $V$ is constant also in the left lead, i.e.\ on $(-\infty,x_2]$ with some $x_2<0$. If $V(x_2)>E$, the injected wave would still be fully reflected, leading to a situation that is qualitatively very similar to the present case. In order to simplify the subsequent proofs, we shall stick here to Assumption \ref{ass:coeff_a} \ref{item:pot_lext_tp}. The extension to a constant potential on $(-\infty,x_2]$ will be discussed in a follow-up work.

\begin{figure}[!ht]
	\includegraphics[trim=3.4cm 18.5cm 6cm 3cm,scale=1]{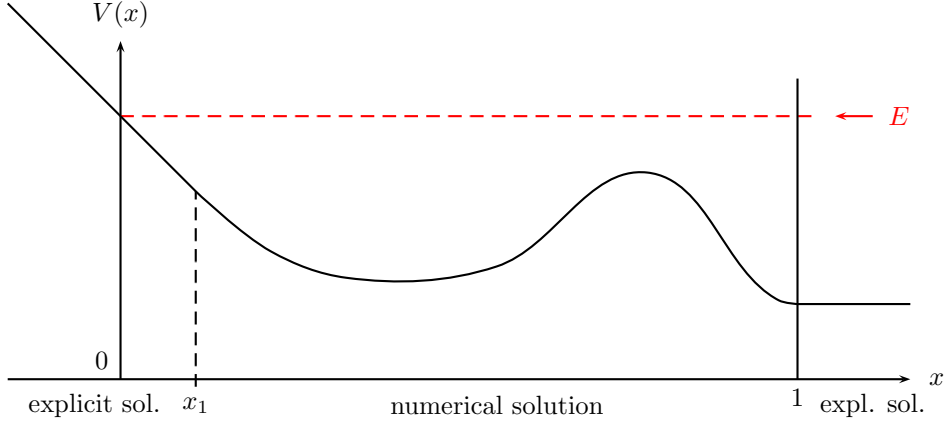}
	\caption{Sketch of the model described in Assumption \ref{ass:coeff_a} with linear potential left of $x_1$. Electrons are injected from the right boundary $x=1$ and there is a turning point of first order at the left boundary $x=0$. The coefficient function is  $a(x) := E-V(x)$.
	The explicit solution form is available for $x\le x_1$ and for $x\ge1$; on $(x_1,1)$ the solution is obtained numerically.}
	\label{fig:potential}
\end{figure}


For each fixed $\eps$, the wave function $\psi$ is $\mathcal{C}^2$ and we require the \emph{scattering solution\/} to be bounded w.r.t.\ $x\in\R$. Since the potential grows linearly for $x<0$, $\psi(x)$ has to decay to $0$ as $x\to-\infty$. Hence, $\psi$ is a scaled Airy function $\Ai$ for $x\le x_1$:
\begin{equation}\label{eqn:psiminus}
\psi_-(x) = c_0\Ai(-\tfrac{x}{\eps^{2/3}})\,, \quad x\in(-\infty,x_1]\,,
\end{equation}
with some $c_0\in\C$ to be determined.
Moreover $\psi$ is a superposition of two \todo{plane} waves \todo{(incoming and outgoing)} for $x\geq 1$:

\begin{equation}\label{eqn:psiplus}
\psi_+(x) = c_1 e^{\frac{\im \sqrt{a_1}}{\eps} x} + c_2 e^{-\frac{\im \sqrt{a_1}}{\eps} x}\,, \quad x\in[1,\infty)\,,
\end{equation}
with $a_1 := a(1)$.

This whole-space problem (with $x\in\R$) can be written as an equivalent BVP for $\psi$ on the interval $[x_1,1]$ by using transparent boundary conditions (BCs) that correspond to the $\mathcal{C}^1$--continuity of the matched whole-space solution. The inhomogeneous transparent BC at $x=1$ is well known from \cite{LK90,ABN11} (see \eqref{bvp:mod2}) and ensures that there is no reflection induced by the BC for an incident wave coming from the right.  $\mathcal{C}^1$--matching of $\psi$ with $\psi_-$ at $x_1$ reads:
\begin{eqnarray}\label{eqn:bc_mod2}
\begin{cases}
\psi(x_1) = c_0 \Ai\left(-\tfrac{x_1}{\eps^{2/3}}\right)\;,\\
\eps\psi^\prime(x_1) = -c_0 \eps^{\frac{1}{3}} {\Ai}'\left(-\tfrac{x_1}{\eps^{2/3}}\right)\;,
\end{cases}
\end{eqnarray}
with some $c_0(\eps)\in\C\setminus{\{0\}}$. Eliminating the (so far) unknown constant $c_0$, the last two conditions are combined into a Robin BC. In summary this yields the following BVP:
\begin{eqnarray}\label{bvp:mod2}
\begin{cases}
\eps^2 \psi^{\prime\prime}(x) +a(x)\psi(x) = 0\;, \quad x\in(x_1,1)\;,\\
\eps^\frac{2}{3}\Ai\left(-\tfrac{x_1}{\eps^{2/3}}\right)\psi^\prime(x_1)+{\Ai}^\prime\left(-\tfrac{x_1}{\eps^{2/3}} \right) \psi(x_1) = 0\;,\\
\eps\psi'(1) - \im\sqrt{a(1)}\psi(1) = -2\im\sqrt{a(1)}\;.
\end{cases}
\end{eqnarray}
Here we already used the assumption that the incident wave has amplitude $1$, and more precisely that $c_2 = e^\frac{\im \sqrt{a_1}}{\eps}$. Since the wave is fully reflected, we have $|c_1|=1$ for the reflection coefficient $c_1$. Hence, the wave $\psi_{+}$ from \eqref{eqn:psiplus} has the maximum amplitude $2$ on $x\geq 1$. The plots in Figures \ref{fig:tp_non_epsunif_psi} and \ref{fig:tp_non_epsunif_psi2} illustrate this.
\medskip

For the solvability of this BVP the following simple result holds: 
\vspace{0.2cm}
\begin{proposition}\label{prop:uniqueness}
Let $x_1\in[0,1)$ and $a\in\mathcal{C}[x_1,1]$ with $a(1)>0$.\
Then the BVP \eqref{bvp:mod2} has a unique solution $\psi\in\mathcal{C}^1[x_1,1]$.
\begin{proof}
This proof is analogous to Proposition 2.3 of \cite{BDM97} and Proposition 1.1 of \cite{AN18}: multiplying the Schr\"odinger equation by $\bar{\psi}$, integrating by parts, and taking the imaginary part.
\end{proof}
\end{proposition}


\emph{\bf Notation and assumptions:\/}\\
Now we recall some notation and assumptions needed to apply the WKB-marching method from \cite{ABN11} to the BVP \eqref{bvp:mod2}. The well-known WKB-approximation (cf.\ \cite{LL85}, \S15 of \cite{Hal13}), for the oscillatory regime where $a(x)\geq \tau_1>0$, is based on inserting the asymptotic power series ansatz 
\begin{equation}\label{eqn:wkb_exp}
\psi(x)\sim \exp\left(\frac{1}{\eps}\sum_{p=0}^\infty\eps^p\phi_p(x)\right)\,
\end{equation}
into the equation \eqref{eqn:1d_schroed}, and comparing $\eps$--powers to successively obtain the functions $\phi_p(x)$. Truncating the sum in the exponential after $p=2$ leads to the $2^{nd}$ order asymptotic WKB-approximation for the oscillatory regime
\begin{equation}\label{eqn:wkb2}
\varphi_2(x) = \frac{\exp\left(\pm\frac{\im}{\eps}\phi(x) \right)}{\sqrt[4]{a(x)}}\;,
\end{equation}
with the phase
\begin{equation}\label{eqn:phase_beta}
\phi(x):=\int_{x_1}^x\left(\sqrt{a(\tau)} - \eps^2\beta(\tau) \right)d\tau\;,\quad \beta(x):=-\frac{(a^{-1/4})^{\prime\prime}}{2a^{1/4}}\;.
\end{equation}
In the WKB-marching method from \cite{ABN11} this $2^{nd}$ order WKB-approxi\-mation is used to transform the equation \eqref{eqn:1d_schroed} to a smoother problem that is then numerically solved on a coarse grid, accurately and efficiently. This is done by reformulating the BVP \eqref{bvp:mod2} into an IVP using the boundary condition at $x_1$ and then scaling the numerical approximation to this IVP (obtained by the scheme we will recall in Section \ref{subsec:wkb_meth}) to also satisfy the boundary condition at $x=1$.

We need to make an assumption to assure the feasibility of the WKB-marching method:
\begin{ass}\label{ass:a_eps_wkb}
Let $a\in \mathcal{C}^5[x_1,1]$ be real valued and satisfy the following bounds
\begin{equation*}
0<\tau_1\leq a(x)\leq \tau_2\;,\quad \forall x\in[x_1,1]\;.
\end{equation*}
Moreover let $0<\eps\leq\eps_0$, with some $\eps_0$ such that
\begin{equation*}
0<\eps_0< \min \left\{ 1,\min_{x_1\leq x \leq 1} [a(x)^{1/4}\beta_+(x)^{-1/2}]\right\}\;,
\end{equation*}
where $\beta_+$ denotes the non-negative part of $\beta$.
\end{ass}
Mind that this assumption on $\eps$ guarantees that the phase $\phi(x)$ for the $2^{nd}$ order WKB-approximation is strictly increasing since the integrand $\sqrt{a}-\eps^2\beta$ is then positive.


\section{Analytical problem: reformulation as IVP}\label{sec:ana_prob}

For the numerical solution of \eqref{bvp:mod2} we want to apply the WKB-marching method from \cite{ABN11} on the interval $[x_1,1]$. To this end we need to reformulate the BVP \eqref{bvp:mod2} as an IVP, whose solution $\hat{\psi}$ will be scaled afterwards to satisfy the BCs in \eqref{bvp:mod2}. Note that the left BC at $x_1$ in \eqref{bvp:mod2} is invariant under scalings.

\underline{\it Step 1:} Since \eqref{bvp:mod2} only includes one condition at $x_1$, it is necessary to prescribe an additional, auxiliary initial value for $\hat\psi(x_1)$. The condition on $\eps\hat{\psi}^\prime(x_1)$ then follows from the Robin BC at $x_1$ in \eqref{bvp:mod2}. On the one hand the two ICs should have the structure of  \eqref{eqn:bc_mod2} (with an appropriate choice of $c_0$), and on the other hand the scaling constant $c_0$ should be of order\footnote{``Big Theta'' is defined as: $f(x)=\Theta(g(x))$ as $x\to a :\Leftrightarrow \exists\,k_1, k_2, \delta>0:\forall x:|x-a|<~\delta: \,k_1|g(x)|\leq |f(x)|\leq k_2|g(x)|$, i.e.\ $g(x)$ is an asymptotic tight bound for $f(x)$. Occasionally we will give $\Theta$ a subscript to specify the variable of asymptotic limit.} $\Theta(\eps^{-\frac{1}{6}})$, such that the initial condition vector $(\hat{\psi}(x_1),\eps\hat{\psi}^\prime(x_1))^\top$ is $\eps$--uniformly bounded above and below. In fact, $c_0= \eps^{-\frac{1}{6}}$ in \eqref{eqn:bc_mod2} yields
\begin{equation}\label{eqn:unif_bnd_ic_ivp}
c\leq\left\|\mtrx{c}{\hat{\psi}(x_1)\\\eps\hat{\psi}^\prime(x_1)}\right\|
=\left\|\mtrx{c}{\eps^{-\frac{1}{6}}\Ai(-\tfrac{x_1}{\eps^{2/3}})\\ -\eps^{\frac{1}{6}}{\Ai}'(-\tfrac{x_1}{\eps^{2/3}})}\right\|
\leq C\;, 
\end{equation}
where the constants $c,C>0$ are independent of $\eps$. This can be verified using the asymptotic expansions for $\Ai(-z)$ and ${\Ai}^\prime(-z)$ from \eqref{eqn:full_asy_airy}: We get the asymptotic representations with the argument $z=\frac{x}{\eps^{2/3}}$ for some (fixed) $x>0$ as $\eps\to 0$:
\begin{equation}\label{eqn:asy_rep_airy}
\begin{aligned}
\Ai(-\tfrac{x}{\eps^{2/3}}) &=  \eps^\frac{1}{6} x^{-\frac{1}{4}}\pi^{-\frac{1}{2}} \left(\cos(\xi(x)) + \tfrac{5}{48}\sin\left(\xi(x)\right)x^{-\frac{3}{2}}\,\eps + \calO(\eps^2) \right)\;,\\
{\Ai}^\prime(-\tfrac{x}{\eps^{2/3}}) &= \eps^{-\frac{1}{6}} x^{\frac{1}{4}}\pi^{-\frac{1}{2}} \left(\sin(\xi(x)) + \tfrac{7}{48}\cos\left(\xi(x)\right)x^{-\frac{3}{2}}\,\eps + \calO(\eps^2)\right)\;,
\end{aligned}
\end{equation}
where $\xi(x):= \frac{2x^\frac{3}{2}}{3\eps}- \frac{\pi}{4}$. We verify

\begin{equation*}
\left\|\begin{array}{c} \eps^{-\frac{1}{6}}\Ai(-\frac{x_1}{\eps^{2/3}})\\ -\eps^{\frac{1}{6}}{\Ai}'(-\frac{x_1}{\eps^{2/3}}) 
\end{array}\right\|^2 = \frac{|\cos(\xi(x_1))|^2+x_1|\sin(\xi(x_1))|^2+\calO(\eps)}{\pi\,x_1^{{1/2}}}\,,\quad\eps\to 0\,.
\end{equation*}
The $\eps$--uniform lower bound on the IC can be found, as $\cos$ and $\sin$ never vanish simultaneously. Hence, this scaling gives a natural balance of $\psi_{-}$ and $\eps\psi_{-}^\prime$. We shall thus consider the IVP
\begin{equation}\label{ivp:tp_mod2}
\begin{cases}
\eps^2\hat\psi^{\prime\prime}(x) + a(x) \hat\psi(x) = 0\;, \quad x\in(x_1,1)\;,\\
\hat\psi(x_1) = \eps^{-\frac{1}{6}}\Ai(-\frac{x_1}{\eps^{2/3}})\;, \\
\eps\hat\psi^\prime(x_1) = -\eps^{\frac{1}{6}}{\Ai}^\prime(-\frac{x_1}{\eps^{2/3}})\;.
\end{cases}
\end{equation}
Here and in the sequel we use the notation $\hat{\psi}$ to refer to the solution of this IVP. 

\underline{\it Step 2:} Next the solution $\hat{\psi}$ of this IVP is scaled as 
\begin{equation}\label{eqn:psi_scale}
\psi(x) := ~\alpha\,\hat{\psi}(x)\,,
\end{equation}
with
\begin{equation}\label{eqn:alpha_scal}
\alpha(\hat{\psi}(1),\hat{\psi}^\prime(1)) := \frac{-2\im\sqrt{a(1)}}{\eps\hat\psi^\prime(1)-\im\sqrt{a(1)}\hat\psi(1)}\;,
\end{equation} 
in order to satisfy the BC at $x=1$. Note that this scaling preserves the BC at $x_1$, and thus, $\psi$ is a solution to the BVP \eqref{bvp:mod2}. From here on we use the notation $\psi$ to refer to the solution of BVP \eqref{bvp:mod2}.
This scaling is also applied to the extension of the solution to $[0,x_1]$ as $\psi(x) = \alpha\,\psi_{-}(x)$, with the choice $c_0=\eps^{-\frac{1}{6}}$ in \eqref{eqn:psiminus}. This equivalence of the BVP to an IVP with a-posteriori scaling was already used in \cite[\S 2]{ABN11} and \cite[Prop. 2.2]{AN18} for closely related problems. 


\emph{\bf The vector valued system:\/}\\
Following \cite{ABN11} it is convenient to reformulate the second order differential equation \eqref{eqn:1d_schroed} as a system of first order. This is done in the following non-standard way: Instead of the vector $(\hat{\psi}(x),\eps\hat{\psi}^\prime(x))^\top$ we shall use 
\begin{equation}\label{eqn:trafo_system_W}
\hat{W}(x)= \mtrx{c}{\hat{w}_1(x)\\ \hat{w}_2(x)}:= \mtrx{c}{a^{1/4}\hat{\psi}(x)\\\frac{\eps(a^{1/4}\hat{\psi})^\prime(x)}{\sqrt{a(x)}}}\;,
\end{equation}
with the transformation matrix
\begin{equation}\label{eqn:trafo_Ax}
A(x) := \mtrx{cc}{a^{1/4}(x) & 0\\ \frac{\eps}{4}a^{-5/4}(x)a^\prime(x) & a^{-1/4}(x)}\,, \quad\textit{i.e.}\quad \hat{W}(x) = A(x)\mtrx{c}{\hat{\psi}(x)\\\eps\hat{\psi}^\prime(x)}\,.
\end{equation}
Under Assumption \ref{ass:a_eps_wkb} (i.e.\ $a(x)$ is bounded away from zero), the transformation matrix $A(x)$ and its inverse are uniformly bounded w.r.t.\ $x\in[x_1,1]$ and $\eps$. Hence, the norms of the two vectors $\hat{W}(x)$ and $(\hat{\psi}(x),\eps\hat{\psi}^\prime(x))^\top$ are equivalent, uniformly in $\eps$. 
After the transformation \eqref{eqn:trafo_system_W}, the IVP \eqref{ivp:tp_mod2} reads
\begin{equation}\label{ivp:tp_mod2_W}
\begin{cases}
\hat{W}^\prime(x) = \left[\frac{1}{\eps}A_0(x)+\eps A_1(x) \right]\hat{W}(x)\;, &\quad x\in[x_1,1]\;,\\
\hat{W}(x_1)=A(x_1)\mtrx{c}{\hat{\psi}(x_1)\\ \eps\hat{\psi}^\prime(x_1)}\in\R^2\;,
\end{cases}
\end{equation}
with the two matrices 
\begin{equation}\label{eqn:mtrx_a0a1}
A_0(x) := \sqrt{a(x)}\mtrx{cc}{0&1\\-1&0}\;;\quad A_1(x):=\mtrx{cc}{0&0\\2\beta(x)&0}\;.
\end{equation}


In order to show (in Section \ref{sec:numeth_errana}) that the WKB-marching method from \cite{ABN11} applied to \eqref{ivp:tp_mod2} yields a uniformly accurate scheme for the BVP \eqref{bvp:mod2}, we shall need $\eps$--uniform boundedness of the scaling factor $\alpha$ from \eqref{eqn:alpha_scal}. This can be inferred from a uniform lower bound on $(\hat{\psi},\eps\hat{\psi^\prime})^\top$, which we establish similarly to \cite[Lemma 3.4]{AN18}:
\begin{lemma}\label{lem:unif_bnd_W}
Let $a(x)\in\mathcal{C}^2[x_1,1]$ and $a(x)\geq\tau_1>0$. Let $\hat\psi(x)$ be the solution to the IVP \eqref{ivp:tp_mod2}, then $(\hat\psi(x),\eps\hat{\psi}^\prime(x))$ is uniformly bounded above and below, i.e.\
\begin{equation}\label{eqn:ul_unif_bnd_psi}
C_1\leq \left\|(\hat{\psi}(x),\eps\hat{\psi}^\prime(x))\right\|\leq C_2\;,\quad x\in[x_1,1]\;,
\end{equation} 
or equivalently 
\begin{equation}\label{eqn:ul_unif_bnd_W}
C_3\leq \|\hat{W}(x)\|\leq C_4\;,\quad x\in[x_1,1]\;,
\end{equation} 
where the constants $C_1,\ldots,C_4>0$ are independent of $0<\eps<\eps_0$.
\begin{proof}
Let $\hat{W}(x)$ be a solution to \eqref{ivp:tp_mod2_W}. A short calculation for the norm $\|\hat{W}(x)\|^2$
$=|\hat{w}_1(x)|^2+|\hat{w}_2(x)|^2$ shows
\begin{equation*}
\left|\frac{d}{dx}\|\hat{W}(x)\|^2\right| = |2\eps\beta(x)(\hat{w}_1\bar{\hat{w}}_2+\bar{\hat{w}}_1\hat{w}_2)|\leq 2\eps|\beta(x)|\|\hat{W}(x)\|^2\;,
\end{equation*}
which implies 
\begin{equation*}
\|\hat{W}(x_1)\|e^{-\eps\int_{x_1}^x|\beta(y)| dy}\leq\|\hat{W}(x)\|\leq \|\hat{W}(x_1)\|e^{\eps\int_{x_1}^x|\beta(y)| dy}\;,\quad x\in[x_1,1]\,.
\end{equation*}
As the norms of $\hat{W}(x)$ and $(\hat{\psi},\eps\hat{\psi}^\prime)^\top$ are ($\eps$--uniformly) equivalent, the proof is concluded if the norm of the initial condition $\hat{W}(x_1)$ is $\eps$--uniformly bounded from above and below. This is again equivalent to \eqref{eqn:unif_bnd_ic_ivp}, proving the assertion.
\end{proof}
\end{lemma}

A solution $\hat{W}$ to the analytical IVP \eqref{ivp:tp_mod2_W} needs to be scaled such that, after transforming back via $A(x)^{-1}$ to $(\psi,\eps\psi^\prime)^\top$, it fits both boundary conditions in \eqref{bvp:mod2}. In analogy to \eqref{eqn:alpha_scal} this is done via the scaling parameter $\tilde{\alpha}\in\C$ defined as
\begin{equation} \label{eqn:alpha_scal_W}
\tilde{\alpha}(\hat{W}(1)) := \frac{-2\im a(1)^{1/4}}{\hat{w}_2(1) - \left[\im + \frac{\eps}{4}a(1)^{-3/2}a^\prime(1)\right]\hat{w}_1(1)} = \alpha(\hat{\psi}(1),\hat{\psi}^\prime(1))\;.
\end{equation}

Now we can write the exact solution $\psi(x)$ to the BVP \eqref{bvp:mod2} extended to the region $[0,1]$ as
\begin{equation*}
\psi(x) = 
\begin{cases}
\tilde{\alpha}(\hat{W}(1))\,\psi_-(x)\;,\qquad & x\in [0,x_1]\,,\\
\tilde{\alpha}(\hat{W}(1))\,\hat{\psi}(x) = \tilde{\alpha}(\hat{W}(1))\,\hat{w}_1(x)\,a(x)^{-\frac{1}{4}}\;,\qquad & x\in [x_1,1]\,.
\end{cases}
\end{equation*}

$\hat{W}$, as well as $\hat{\psi}$, are real-valued, and $\psi$ only becomes complex-valued due to the scaling with $\tilde{\alpha} (=\alpha)$. Since $\hat{W}\in\R^2$, the denominator in $\tilde{\alpha} $ cannot vanish, except for the trivial solution $\hat{W}\equiv 0$. Therefore the scaling by $\tilde{\alpha}$ is well-defined and one can show the following properties.

\begin{lemma}[see Lemma 3.7 in \cite{AN18}]\label{lem:alpha_lip_bnd}
Let $\delta>0$ be fixed. Then the map $\tilde{\alpha}: \R^2\setminus B_\delta(0)\rightarrow \C$ in \eqref{eqn:alpha_scal_W} is Lipschitz continuous with Lipschitz constant $L_{\tilde{\alpha}}>0$ and bounded with a constant $C_{\tilde{\alpha}} $. Both, $L_{\tilde{\alpha}}$ and $C_{\tilde{\alpha}}$, can be chosen uniformly with respect to $0<\eps\leq \eps_0$.
\end{lemma}



\section{Asymptotic blow-up at the turning point}\label{sec:blow_up}

The goal of this work is to construct an $\eps$--uniformly accurate numerical scheme for \eqref{bvp:mod2}. Since it incorporates (implicitly) the turning point at $x=0$, it shall be a generalization of \cite{ABN11,AN18}. A key ingredient of the numerical analysis in these papers was the uniform boundedness of the solution $\psi$ w.r.t.\ $\eps$. But when including a turning point, this does \emph{not\/} hold any more, which is a main challenge of the situation at hand. At the turning point, solutions to the BVP \eqref{bvp:mod2} exhibit blow-up behavior as $\eps\to 0$, i.e.\ $|\psi(0)| = \Theta(\eps^{-\frac{1}{6}})$. This is shown in the following explicitly solvable example and the proposition that follows it.

\begin{exam}\label{exam:airy}

Consider \eqref{bvp:mod2} with $x_1=0$ and $a(x) = x$ for $x\in[0,1]$ and $0<\eps<1$. Then the explicit solution reads
\begin{equation}\label{eqn:solairy_bvp_mod2}
\psi_\eps(x) = \frac{2}{\eps^{-{1/6}}\Ai(-\tfrac{1}{\eps^{2/3}})-\im\eps^{1/6}{\Ai}'(-\tfrac{1}{\eps^{2/3}})}\eps^{-\frac{1}{6}}\Ai\left(-\tfrac{x}{\eps^{2/3}}\right)\;.
\end{equation}
Figures \ref{fig:tp_non_epsunif_psi} and \ref{fig:tp_epsunif_psiPri} illustrate that $\eps\|\psi^\prime_\eps\|_{L^\infty(0,1)}$ is uniformly bounded w.r.t.\ $0<\eps\leq 1$, but $\|\psi_\eps\|_{L^\infty(0,1)}$ is \emph{not\/} since $\{|\psi_\eps(0)|\}$ becomes unbounded as $\eps\to0$. 
\end{exam}

\begin{figure}[!ht]
	\includegraphics[scale=0.63]{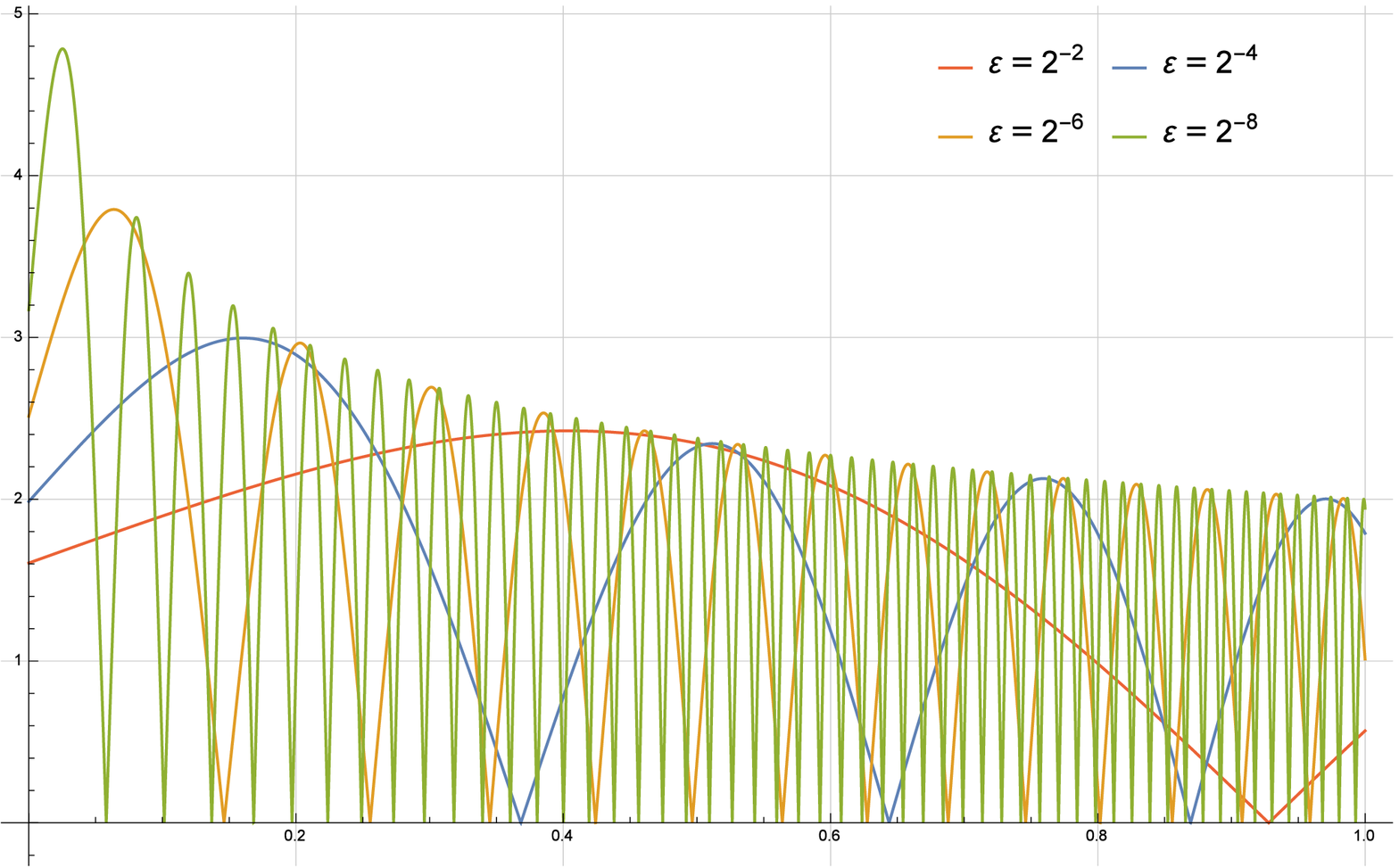}
	\caption{$|\psi_\eps(x)|$ for various values of $\eps$. At the turning point $(x=0)$ it increases with the order $\eps^{-\frac{1}{6}}$.} 
	\label{fig:tp_non_epsunif_psi}
\end{figure}
\begin{figure}[!ht]
	\includegraphics[scale=0.63]{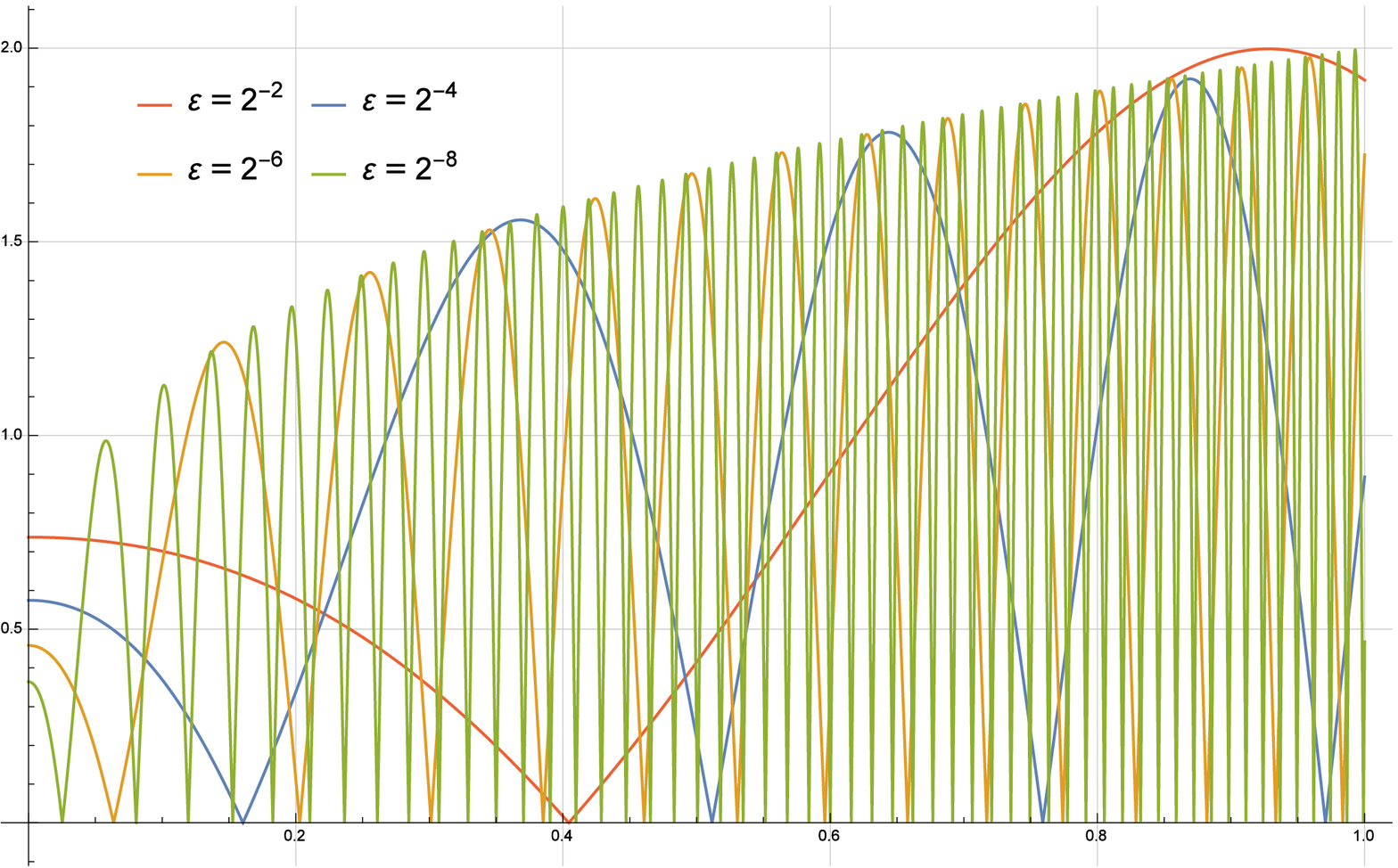}
	\caption{$\eps|\psi_\eps^\prime(x)|$ for various values of $\eps$.}
	\label{fig:tp_epsunif_psiPri}
\end{figure}

This blow-up and, resp., boundedness behavior of $\psi_\eps$ actually extends to all potentials satisfying Assumption \ref{ass:coeff_a}:

\begin{proposition}\label{prop:bnd_mod2}
Let $x_1\in(0,1)$ and $a(x)$ be as in Assumption \ref{ass:coeff_a} and $\mathcal{C}^2$ on $[x_1,1]$. Then the family of solutions $\{\psi_\eps(x)\}$ to the BVP \eqref{bvp:mod2}, extended with the Airy solution on $[0,x_1]$, satisfies:
\begin{enumerate}[label = \alph*)]
\item{
$\|\psi_\eps\|_{L^\infty(0,1)}$ is of the (sharp) order $\Theta(\eps^{-\frac{1}{6}})$ for $\eps\to0$.
}\label{item:nonunifbnd_mod2_psi}
\item{
$\eps\|\psi_\eps^\prime\|_{L^\infty(0,1)}$ is uniformly bounded with respect to $\eps\to0$.
}\label{item:unifbnd_mod2_psiPri}
\end{enumerate}
\begin{proof} 
For readability of this proof, we omit the index $\eps$ in $\psi_\eps$ and $\hat{\psi}_\eps$.
As discussed in Section \ref{sec:ana_prob}, the solution $\psi$ to the BVP \eqref{bvp:mod2} can be obtained by scaling the solution $\hat{\psi}$ to the IVP \eqref{ivp:tp_mod2} with the constant $\alpha$ from \eqref{eqn:alpha_scal}, i.e.\ $\psi = \alpha\,\hat{\psi}$. With some (fixed) $x_0\in(0,x_1)$ we make a case distinction:

\noindent{\it\underline{Region $x_0\leq x\leq 1$:}}\\
For $x\ge x_0$ we consider the IVP \eqref{ivp:tp_mod2_W} first with a generic initial condition at $x_0$. For such $x$ we have $a(x)\geq\tau_3>0$, and therefore the transformation matrices $A(x)$ and its inverse are uniformly bounded w.r.t.\ $x$ and $\eps$. Denoting by $W$ the vector valued solution to this IVP, this matrix bound implies equivalence (uniform in $\eps$) of the norms of the two vectors ${W}(x)$ and $({\psi}(x),\eps{\psi}^\prime(x))^\top$. Note that this equivalence would not hold for the choice $x_0=0$. Analogously as in the proof of Lemma \ref{lem:unif_bnd_W} we obtain 
\begin{equation}\label{eqn:ul_est_hatW_x0}
\|{W}(x_0)\|e^{-\eps\int_{x_0}^x|\beta(y)| dy}\leq\|{W}(x)\|\leq \|{W}(x_0)\|e^{\eps\int_{x_0}^x|\beta(y)| dy}\,,\quad 0<x_0\leq x\,.
\end{equation}

{\it\underline{ Step 1:}} 
First we consider an auxiliary problem that corresponds to the ``pure'' Airy function solution: Let $\tilde W$ be the solution of the IVP \eqref{ivp:tp_mod2_W} on $x\ge x_0$ with $\eps =1$, $a(x) = x$ and the initial condition 
\begin{equation*}
  \tilde{W}(x_0) = A_{\eps=1}(x_0) \mtrx{c}{\Ai(-x_0)\\ -{\Ai}^\prime(-x_0)} \,,
\end{equation*}
obtained from \eqref{ivp:tp_mod2}. The solution reads
\begin{equation*}
\tilde{W}(x) = \mtrx{c}{a^{1/4}\tilde{\psi}(x)\\\frac{(a^{1/4}\tilde{\psi})^\prime(x)}{\sqrt{a(x)}}} = \mtrx{c}{x^\frac{1}{4}\Ai(-x) \\ \frac{1}{4}x^{-\frac{5}{4}}\Ai(-x)-x^{-\frac{1}{4}}{\Ai}^\prime(-x)} \,,\quad x \in [x_0,\infty)\,.
\end{equation*}
Then the estimates \eqref{eqn:ul_est_hatW_x0} with $\int_{x_0}^x|\beta(y)|dy = \frac{5}{48}(x_0^{-\frac{3}{2}} - x^{-\frac{3}{2}})\le \frac{5}{48} x_0^{-\frac{3}{2}}$ yield
\begin{equation}\label{eqn:ul_est_tildeW_x0}
d_{x_0}\|\tilde{W}(x_0)\|\leq \|\tilde{W}(x) \|\leq c_{x_0}\|\tilde{W}(x_0)\|\;,\quad x\in[x_0,\infty)\,,
\end{equation}
with the constants defined as 
\begin{equation*}
  d_{x_0}:=e^{-\int_{x_0}^\infty|\beta(y)| dy} \,,\quad
  c_{x_0}:= e^{\int_{x_0}^\infty|\beta(y)| dy} \,.
\end{equation*}

{\it\underline{ Step 2:}} 
Let $\hat W$ be the solution of the IVP \eqref{ivp:tp_mod2_W} on $x_0\le x\le x_1$ with $0<\eps \le1$, $a(x) = x$ and the initial condition 
\begin{equation*}
  \hat{W}(x_0) = A_{\eps}(x_0) \mtrx{c}{\eps^{-\frac{1}{6}}\Ai(-\tfrac{x_0}{\eps^{2/3}})\\ -\eps^{\frac16}\Ai'(-\tfrac{x_0}{\eps^{2/3}})} \,.
\end{equation*}
Then the vector valued solution corresponds to the scaled Airy function:
\begin{equation}\label{eqn:hatW}
\hat{W}\left(x\right) = \mtrx{c}{
x^\frac{1}{4}\eps^{-\frac{1}{6}}\Ai\left(-\frac{x}{\eps^{2/3}}\right) \\
 \frac{1}{4}x^{-\frac{5}{4}}\eps^\frac{5}{6}\Ai\left(-\frac{x}{\eps^{2/3}}\right)-x^{-\frac{1}{4}}\eps^\frac{1}{6}{\Ai}^\prime\left(-\frac{x}{\eps^{2/3}}\right)}
 = \tilde{W}\left(\tfrac{x}{\eps^{2/3}}\right) \,,\quad x_0\leq x\leq x_1\,,
\end{equation}
where $x/\eps^{\frac23}\ge x_0$. Hence, \eqref{eqn:ul_est_tildeW_x0} implies
\begin{equation}\label{eqn:ul_hatWscal}
d_{x_0}\|\tilde{W}(x_0)\| \leq\|\hat{W}\left(x\right)\|\leq c_{x_0}\|\tilde{W}(x_0)\| \;,\quad x_0\leq x\leq x_1\,.
\end{equation}
Since $\|\tilde{W}(x_0)\|$ is independent of $\eps$, this proves $\|\hat{W}(x)\| = \Theta_\eps(1)$ on $[x_0,x_1]$.

{\it\underline{ Step 3:}} 
On $x_1\leq x\leq 1$ the (generic) function $a(x)$ satisfies $a(x)\geq\tau_1>0$. Hence, we obtain the following upper and lower bounds analogously to \eqref{eqn:ul_est_hatW_x0}:
\begin{equation}\label{eqn:ul_est_hatW_x1}
d_{x_1}\|\hat{W}(x_1)\|\leq\|\hat{W}(x)\|\leq c_{x_1}\|\hat{W}(x_1)\|\;,\quad x_1\leq x\leq 1\,,
\end{equation}
where $d_{x_1}:=e^{-\int_{x_1}^1|\beta(y)| dy}$ and $c_{x_1}:= e^{\int_{x_1}^1|\beta(y)| dy}$. Since \eqref{eqn:ul_hatWscal} particularly holds for $\|\hat{W}(x_1)\|$, the bounds in \eqref{eqn:ul_est_hatW_x1} yield $\|\hat{W}(x)\| = \Theta_\eps(1)$ on $[x_1,1]$.
Since the norms of $\hat{W}(x)$ and $(\hat{\psi}(x),\eps\hat{\psi}^\prime(x))^\top$ are ($\eps$ uniformly) equivalent, we get 
\begin{equation}\label{eqn:ul_hatPsi_theta}
\left\|\mtrx{c}{\hat{\psi}(x)\\ \eps\hat{\psi}^\prime(x)}\right\| = \Theta_\eps(1)\,,\quad x_0\leq x\leq 1\,.
\end{equation}
This yields the asymptotic behavior of the scaling constant $\alpha$ from \eqref{eqn:alpha_scal}: \begin{equation}\label{eqn:asy_alpha}
\alpha(\hat{\psi}(1),\hat{\psi}^\prime(1)) = \Theta_\eps(1)\,.
\end{equation}
In the region $x_0\leq x\leq1$, this yields for the vector solution of the BVP \eqref{bvp:mod2}
\begin{equation}\label{eqn:unif_bnd_notp}
\left\|\mtrx{c}{\psi(x)\\ \eps\psi^\prime(x)}\right\| = |\alpha|\left\|\mtrx{c}{\hat{\psi}(x)\\ \eps\hat{\psi}^\prime(x)}\right\| = \Theta_\eps(1)\,,\quad x_0\leq x\leq 1\,.
\end{equation}

\noindent{\it\underline{Region $0\leq x\leq x_0$:}}\\
\indent{\it\underline{ Step 4:}} 
Note that the solution to the BVP \eqref{bvp:mod2}, extended to $[0,1]$, exhibits an asymptotic blow-up at the turning point $x=0$: As  $\hat{\psi}(x)= \eps^{-\frac{1}{6}}\Ai(-\frac{x}{\eps^{2/3}})$ on $[0,x_1]$, it holds that
\begin{equation}\label{eqn:asy_blowup_tp}
|\psi_\eps(0)| = |\alpha\,\eps^{-\frac{1}{6}}\Ai(0)| = \Theta(\eps^{-\frac{1}{6}})\,.
\end{equation} 
Moreover, the following proves that $\max_{x\in[0,x_0]}|\psi_\eps(x)| = \Theta(\eps^{-\frac{1}{6}})$:
The Airy function $\Ai$ is continuous and bounded on $\R$. It attains its unique maximum at some $y_{max}^{\Ai}\approx -1.01879$ , i.e.\  
\begin{equation*}
\argmax_{y\in\R}\Ai\left(y\right) = y_{max}^{\Ai}\;.
\end{equation*} 
The maximum of $\Ai(-\frac{x}{\tilde{\eps}^{2/3}})$, located at $x_{max}^{\Ai} = ~{-\tilde{\eps}^{2/3}y_{max}^{\Ai}}$, lies inside $[0,x_0]$ for $\tilde{\eps}$ sufficiently small. Hence, $\max_{x\in[0,x_0]} \Ai(-\frac{x}{\eps^{2/3}}) =  M:=\Ai\left(y_{max}^{\Ai}\right)$ is constant for $0<\eps\leq\tilde{\eps}$.
Since $\alpha=\Theta_\eps(1)$ we get
\begin{equation}\label{eqn:psi_norm_blowup}
\max_{x\in[0,x_0]}|\psi_\eps(x)|= \alpha\,\eps^{-\frac{1}{6}} M=\Theta(\eps^{-\frac{1}{6}})\;,\quad 0<\eps\leq\tilde{\eps}\;.
\end{equation}
Hence, \eqref{eqn:unif_bnd_notp} together with \eqref{eqn:psi_norm_blowup} yields
\begin{equation*}
\|\psi_\eps\|_{L^\infty(0,1)} = \Theta(\eps^{-\frac{1}{6}})\,,
\end{equation*}
thus proving \ref{item:nonunifbnd_mod2_psi}.

{\it\underline{ Step 5:}} 
To prove the uniform bound on the $\eps$--scaled derivative, we use the asymptotic representation \eqref{eqn:asy_rep_airy} for ${\Ai}^\prime$ and the fact that 
\begin{equation*}
\eps\psi_\eps^\prime(x) = -\alpha\eps^\frac{1}{6}{\Ai}^\prime\left(-\tfrac{x}{\eps^{2/3}}\right)\,.
\end{equation*}
Fix some $z_0>0$ for another case distinction. With $z=\frac{x}{\eps^{2/3}}$ the asymptotic representation \eqref{eqn:asy_rep_airy} holds for $z\geq z_0>0$. I.e.\ there exists a $c_1>0$ such that for $z_0\eps^\frac{2}{3}\leq x\leq x_0$ it holds that 
\begin{equation*}
|\eps\psi_\eps^\prime(x)|=\left|\alpha\eps^\frac{1}{6}{\Ai}^\prime\left(-\tfrac{x}{\eps^{2/3}}\right)\right|\leq \left|\tfrac{\alpha x^\frac{1}{4}}{\sqrt{\pi}} \left(\sin(\xi(x)) + x^{-\frac{3}{2}}\calO(\eps)\right)\right|\leq c_1\,, 
\end{equation*}
since $\alpha = \Theta_\eps(1)$ and $x^{-5/4}\eps \leq z_0^{-5/4}\eps^{1/6}\leq c_2$ for some $c_2>0$. For small arguments $0\leq z\leq z_0$ (i.e.\ close to the turning point) it holds that $|{\Ai}^\prime(-z)|\leq c_3$ for some $c_3>0$, as ${\Ai}^\prime$ is continuous on a compact set.  Hence, for $0\leq x\leq z_0\eps^\frac{2}{3}$ it holds that
\begin{equation*}
|\eps\psi_\eps^\prime(x)|=\left|\alpha\eps^\frac{1}{6}{\Ai}^\prime\left(-\tfrac{x}{\eps^{2/3}}\right)\right|\leq |\alpha|\eps^\frac{1}{6}c_3\leq c_4\,, 
\end{equation*}
for some $c_4>0$. All four constants $c_1,\ldots,c_4$ are independent of $\eps$ yielding \linebreak
$\eps\|\psi_\eps^\prime\|_{L^\infty(0,x_0)}=\calO(1)$ for $\eps\to0$. Together with \eqref{eqn:unif_bnd_notp} this yields an overall uniform bound, i.e.\
\begin{equation*}
\eps\|\psi_\eps^\prime\|_{L^\infty(0,1)} = \calO_\eps(1)\,.
\end{equation*}
\end{proof}
\end{proposition}

\begin{rem}\label{rem:airy}
The proof of the above proposition illustrates the reason for the asymptotic blow-up of $|\psi_\eps(0)|$: Essentially, it stems from the $x^{-1/4}$--decay of the flipped Airy function $\Ai(-x)$ as $x\to\infty$ (note that $\tilde w_1=x^{1/4}\Ai(-x)$ satisfies \eqref{eqn:ul_est_tildeW_x0}). Close to the turning point of \emph{first order\/}, $\psi_\eps$ behaves like the scaled Airy function $\eps^{-1/6}\Ai(-x\,\eps^{-2/3})$. In the scattering model of Section \ref{sec:scat_mod}-\ref{sec:numeth_errana} this even holds exactly, and in Section \ref{sec:mod3} this will hold approximately. This $\eps$--scaling of the $x$ variable compresses this Airy function decay to the (small) interval $[0,x_1]$. At the fixed point $x_1$,  $\Ai(-x_1\eps^{-2/3})$  is proportional to $\eps^{1/6}$, which we compensated by the scaling $\eps^{-{1/6}}$ yielding an $\eps$--uniformly bounded initial condition $\hat{W}(x_1)$.  This $\eps$--uniformity is then not affected any more on the subsequent interval $[x_1,1]$, since $a(x)$ is there uniformly 
bounded away from zero. Hence, the solution propagator is $\eps$--uniformly bounded (above and below) on $[x_1,1]$, see Step 3 in the above proof. On the other hand, at the turning point $x=0$ the Airy function has the value $\Ai(0)$, i.e.\ constant w.r.t.\ $\eps$. Hence, the $\eps^{-{1/6}}$ scaling yields the asymptotic blow-up at $x=0$.
\end{rem}

This asymptotic blow-up is, for the time being, one of the key problems for extending asymptotic preserving schemes (like the WKB-method from \cite{ABN11} or the adiabatic integrators from \cite{LJL05}) up to the turning point. To mitigate this problem, yet still include the turning point into the scattering system, we made the simplifying assumption that $a(x) \equiv x$ on some (small) interval $[0,x_1]$. This way we shall match the analytic solution on $[0,x_1]$ to a numerical solution on $[x_1,1]$. The former is explicit up to a scaling factor that, however, inherits a numerical error from the approximation to $\psi(1)$. 


\section{Numerical method and error analysis}\label{sec:numeth_errana}

In this section we will review the WKB-marching method for the IVP \eqref{ivp:tp_mod2} and derive error estimates for the BVP \eqref{bvp:mod2}, extended to $[0,1]$.

We recall that the BVP \eqref{bvp:mod2} is solved in two steps: First the corresponding IVP \eqref{ivp:tp_mod2} is solved numerically; then the numerical solution is scaled according to \eqref{eqn:psi_scale}, to fit the right boundary condition. For the turning point problem we actually want to solve \eqref{bvp:mod2} on $[0,1]$. But the solution to this BVP on $[0,x_1]$ is given by a scaled Airy function as in \eqref{eqn:psiminus}. We are left with numerically approximating a solution on $[x_1,1]$ and matching the two parts at $x_1$ to obtain a solution on the whole interval. In the $2$--step solution process we incur an error from the WKB-marching method on $[x_1,1]$ and this propagates into a second error from the (inaccurate) $\alpha$--scaling of the ``Airy-solution'' on $[0,x_1]$.

The essential novelty compared to \cite{AN18} is the inclusion of a first order turning point at $x=0$. We proved in Section \ref{sec:blow_up} that the exact solution $\psi_\eps$ blows up at the turning point like $\Theta(\eps^{-\frac{1}{6}})$. Hence, one might expect that the corresponding numerical error would also be unbounded there. We recall that the numerical error stems from the $\alpha$--scaling, where $\alpha(\hat\psi,\hat\psi^\prime)$ depends on the numerically obtained approximations of $\hat\psi(1)$ and $\hat\psi^\prime(1)$. \todo{In fact, the inclusion of the turning point ``costs'' a factor $\eps^{-\frac{1}{6}}$ in the error estimates due to the blow-up of the sequence $\psi_\eps$ at $x=0$ (cp.\ the estimate \eqref{eqn:wkbmm2_errest_W} to \eqref{eqn:errest_01} below). In these estimates the negative power of $\eps$ can be compensated by restricting the step size $h$ in dependence of $\eps$. With a turning point, this $h(\eps)$--relation gets slightly less favorable. }

\todo{In the special case of an explicitly integrable phase $\phi(x)$ from \eqref{eqn:phase_beta},} the WKB-marching method is an \emph{asymptotically correct\/}\footnote{I.e.\ the numerical error decreases to zero with $\eps\to 0$, even for a fixed spatial grid.} scheme w.r.t.\ $\eps$ (with error order $\calO(\eps^3)$). This asymptotic correctness will compensate for the fact that the solution sequence $\psi_\eps$ is unbounded at the turning point $x=0$, and it will \todo{still} yield an overall 
asymptotically correct scheme\footnote{Note: This constitutes a \emph{numerical scheme\/} only for a coefficient function $a(x)$ that is linear (or even 
quadratic, as in Section \ref{sec:mod3} below) on $[0,x_1]$ for some $x_1\in(0,1)$.} for \eqref{bvp:mod2}.

To clarify the notation, we summarize it in the following table. Here the superscript $^{(\prime)}$ means that we refer to the function as well as its derivative. Let $x_1<x_2<\ldots<x_N=1$ be a grid for the numerical method on $[x_1,1]$, where $x_1>0$, and $h:=\max_{2\leq n\leq N}|x_n-x_{n-1}|$ is the step size. 

{\small
\begin{center} 
	\begin{minipage}{\linewidth} 
        \captionof{table}{Summary of notation} 
        \label{tab:notation} 
	\begin{tabular}{|c|l|c|}
      	\hline
        scalar & & domain \\
        \hline
        $\psi^{(\prime)}(x)$ & exact solution to the BVP \eqref{bvp:mod2} & $[0,1]$\\
        & (and extension to $[0,x_1]$) & \\
        $\psi_h^{(\prime)}(x)$ & numerical approximation for the BVP \eqref{bvp:mod2} & $[0,x_1]\cup\{x_1,\ldots,x_N\}$ \\
        & (and extension to $[0,x_1]$) & \\
        $\psi_-^{(\prime)}(x)$ & exact solution to \eqref{eqn:1d_schroed} satisfying the left BC in \eqref{bvp:mod2} & $[0,x_1]$\\
        $\hat{\psi}^{(\prime)}(x)$ & exact solution to the IVP \eqref{ivp:tp_mod2} & $[x_1,1]$ \\
        $\hat\psi^{(\prime)}_{h,n}$ & numerical approximation for the IVP \eqref{ivp:tp_mod2} & $\{x_1,\ldots,x_N\}$ \\
        \hline
        vectorial & & domain \\  
        \hline
        $W(x)$ & exact solution to the BVP \eqref{bvp:mod2} & $[x_1,1]$ \\
        $W_n$ & numerical approximation for the BVP \eqref{bvp:mod2} & $\{x_1,\ldots,x_N\}$ \\
        $\hat{W}(x)$ & exact solution to the IVP \eqref{ivp:tp_mod2_W} & $[x_1,1]$ \\
        $\hat{W}_n$ & numerical approximation for the IVP \eqref{ivp:tp_mod2_W} & $\{x_1,\ldots,x_N\}$ \\
        \hline
   	\end{tabular}
	\end{minipage} 
\end{center} 
 }
Occasionally we shall use a sub- or superscript $\eps$ to emphasize the $\eps$--dependence of that quantity. 

The Airy-WKB scheme for an approximation $\psi_h^{(\prime)}$ to the solution $\psi^{(\prime)}$ of the BVP \eqref{bvp:mod2} extended to $[0,1]$ consists of the following steps:\medskip

\underline{\it Step 1:}
\begin{enumerate}[label=\alph*)]
\item 
On $[0,x_1]$: The solution \eqref{eqn:psiminus} satisfying the ICs in \eqref{ivp:tp_mod2} reads
\begin{equation*}
\psi_-(x) =  \eps^{-\frac{1}{6}}\Ai\left(-\tfrac{x}{\eps^{2/3}}\right)\;;\;\;\eps\psi_-^\prime(x) =  -\eps^{\frac{1}{6}}{\Ai}^\prime\left(-\tfrac{x}{\eps^{2/3}}\right)\,.
\end{equation*}
\item 
On  $[x_1,1]$: Compute a numerical approximation $(\hat{\psi}_{h,n},\eps\hat{\psi}_{h,n}^\prime)^\top$ to the IVP \eqref{ivp:tp_mod2} on the grid $\{x_1,\ldots,x_N\}$ via the WKB-marching method.
\end{enumerate}

\underline{\it Step 2:}\medskip

Scale $\psi_-^{(\prime)}$ and $\hat{\psi}_{h,n}^{(\prime)}$ using $\alpha:=\alpha(\hat{\psi}_{h,N},\hat{\psi}_{h,N}^\prime)$ and set
\begin{equation*}
\psi_h^{(\prime)}(x) :=
\begin{cases}
\alpha\,\psi_-^{(\prime)}(x)\,, & x\in[0,x_1]\;,\\
\alpha\,\hat{\psi}_{h,n}^{(\prime)}(x)\,, & x\in\{x_1,\ldots,x_N\}\;.
\end{cases}
\end{equation*}


\subsection{The WKB-marching method on $[x_1,1]$ }\label{subsec:wkb_meth}

We shall now review the basics of the second order WKB-marching method from \cite{ABN11}:
Here the focus is on the algorithm and error estimates. The background, including motivation for the tools used in this method can be read in \cite{ABN11}. The method consists of two parts, first a transformation of the highly oscillatory problem \eqref{eqn:1d_schroed} to a smoother problem, and second the numerical discretization of said smooth problem as to obtain an $\eps$--asymptotically correct scheme.

\underline{Analytic transformation:}
The first order system \eqref{ivp:tp_mod2_W} for $\hat{W}(x)$ is transformed to the system in the variable $Z(x)$ as follows:
\begin{equation*}
Z(x) := \exp\left(-\tfrac{\im}{\eps} \Phi(x)\right)P\,\hat{W}(x),\quad x\in[x_1,1]\,,
\end{equation*}
with matrices 
\begin{equation*}
P := \frac{1}{\sqrt{2}}\mtrx{cc}{\im & 1 \\ 1 & \im}\;;\quad \Phi(x):=\mtrx{cc}{\phi(x) & 0 \\ 0 & -\phi(x)}\,,
\end{equation*}
where $\phi(x)$ is the phase function defined in \eqref{eqn:phase_beta}. \todo{For this analytic transformation we assume that $\phi$ is explicitly available (for a generalization on numerically computed phases see Remark \ref{rem5.3} below).} This yields the system
\begin{equation}\label{eqn:aban_z}
\begin{cases}
&Z'(x) = \eps N^\eps(x)\,Z(x)\,,\\
&Z(x_1) = P\hat{W}(x_1)\,,
\end{cases}
\end{equation}
where $N^\eps(x)$ is non-zero only in the off-diagonal entries
\begin{equation*}
N^\eps_{1,2}(x) = \beta(x) e^{-\frac{2\im}{\eps}\phi(x)}\,, \quad
N^\eps_{2,1}(x) = \beta(x) e^{\frac{2\im}{\eps}\phi(x)}\,.
\end{equation*}
The above system exhibits much smoother solutions compared to the system for $\hat{W}(x)$ from \eqref{ivp:tp_mod2_W}. Moreover, the strong limit of its solutions $Z_\eps$ as $\eps\to 0$ satisfies the trivial equation $Z^\prime(x) = 0$, since $N^\eps(x)$ is $\eps$--uniformly bounded. Next we recall from \cite{ABN11} a numerical scheme that is at the same time $\eps$--asymptotically correct and second order in the step size $h$.

\underline{Numerical scheme:}
The second order (in $h$) scheme is rather non-standard and developed via the second order Picard approximation of \eqref{eqn:aban_z}: 
\begin{equation*}
\todo{
Z_{n+1} = Z_n+\eps\int_{x_n}^{x_{n+1}}\!\!\!\!\!\!\!\!\!\!\!\!N^\eps(x)\,dx\,Z_n + \eps^2\int_{x_n}^{x_{n+1}}\!\!\!\!\!\!\!\!\!\!\!\!N^\eps(x)\int_{x_n}^{x}\!\!\!\!N^\eps(y)\,dy\,dx\,Z_n\,.
}
\end{equation*}
\todo{These (iterated)} oscillatory integrals \todo{(with $\phi$ assumed to be known exactly)} are then approximated using similar techniques as the \emph{asymptotic method\/} in \cite{INO06}. This yields the following scheme that is $\eps$--asymptotically correct:

For a given initial condition $Z_1:= Z(x_1)$ the algorithm reads
\begin{equation}\label{eqn:wkb_iter}
Z_{n+1} = (I+A_n^1+A_n^2)Z_n\,,\quad n = 1,\ldots,N-1\,,
\end{equation}
with the matrices $A_n^1$ and $A_n^2$ given as
\begin{align*}
& A_n^1 :=\\
 &-\im\eps^2\mtrx{cc}{0 & \!\!\!\!\!\!\!\!\!\!\!\!\!\!\!\!\!\!\!\!\!\!\!\!\beta_0(x_n)e^{-\frac{2\im}{\eps}\phi(x_n)}-\beta_0(x_{n+1})e^{-\frac{2\im}{\eps}\phi(x_{n+1})}\\ \\ \!\!\beta_0(x_{n+1})e^{\frac{2\im}{\eps}\phi(x_{n+1})}-\beta_0(x_n)e^{\frac{2\im}{\eps}\phi(x_n)} & 0 }\\
 & +\eps^3\mtrx{cc}{0 & \!\!\!\!\!\!\!\!\!\!\!\!\!\!\!\!\!\!\!\!\!\!\!\!\beta_1(x_{n+1})e^{-\frac{2\im}{\eps}\phi(x_{n+1})}-\beta_1(x_n)e^{-\frac{2\im}{\eps}\phi(x_n)}\\ \\ \!\!\beta_1(x_{n+1})e^{\frac{2\im}{\eps}\phi(x_{n+1})}-\beta_1(x_n)e^{\frac{2\im}{\eps}\phi(x_n)} & 0 }\\
& +\im\eps^4\beta_2(x_{n+1})\mtrx{cc}{0 & -e^{-\frac{2\im}{\eps}\phi(x_n)}H_1\left(-\frac{2}{\eps}S_n\right)\\ \\  e^{\frac{2\im}{\eps}\phi(x_n)}H_1\left(\frac{2}{\eps}S_n\right) & 0 }\\
& -\eps^5\beta_3(x_{n+1})\mtrx{cc}{0 & e^{-\frac{2\im}{\eps}\phi(x_n)}H_2\left(-\frac{2}{\eps}S_n\right)\\ \\  e^{\frac{2\im}{\eps}\phi(x_n)}H_2\left(\frac{2}{\eps}S_n\right) & 0 }\,,
\end{align*}

\begin{align*}
A_n^2 := &-\im\eps^3(x_{n+1}-x_n)\frac{\beta(x_{n+1})\beta_0(x_{n+1})+\beta(x_n)\beta_0(x_n)}{2}\mtrx{cc}{ 1 & 0 \\ 0 & -1}\\
& - \eps^4\beta_0(x_n)\beta_0(x_{n+1})\mtrx{cc}{ H_1\left(-\frac{2}{\eps}S_n\right) & 0 \\ \\0 & H_1\left(\frac{2}{\eps}S_n\right)}\\
& +\im\eps^5\beta_1(x_{n+1})[\beta_0(x_n)-\beta_0(x_{n+1})]\mtrx{cc}{ H_2\left(-\frac{2}{\eps}S_n\right) & 0 \\ \\ 0 & -H_2\left(\frac{2}{\eps}S_n\right)}\,.
\end{align*}
Here we used the following notations: 
\begin{align*}
\beta_0(y):= \tfrac{\beta}{2(\sqrt{a}-\eps^2\beta)}(y)\,; \quad &\beta_{k+1}(y):=\tfrac{1}{2\phi^\prime(y)}\tfrac{d\beta_k}{dy}(y)\,,\quad k=0,1,2\,;\\ \\
H_1\left(y\right):=e^{\im y}-1\,;\quad &H_2\left(y\right):= e^{\im y}-1-\im y\,,
\end{align*}
and the discrete phase increments are
\begin{equation*}
S_n:= \phi(x_{n+1})-\phi(x_n) = \int_{x_n}^{x_{n+1}}\left(\sqrt{a(\tau)} - \eps^2\beta(\tau)\right)d\tau\,,\quad n=1,\ldots,N-1\,.
\end{equation*}

In the end we obtain a sequence of vectors $Z_n$ which we have to transform back via
\begin{equation}\label{eqn:wkb_back_trafo}
\hat{W}_n = P^{-1}e^{\frac{\im}{\eps}\Phi(x_n)}Z_n\,, \quad n=1,\ldots,N\,.
\end{equation}
This yields an approximation of the solution to the vector valued system \eqref{ivp:tp_mod2_W} for $\hat{W}$.

Now let us formulate a discrete analogue of Lemma \ref{lem:unif_bnd_W}.

\begin{lemma}\label{lem:unif_bnd_Wn}
Let Assumption \ref{ass:a_eps_wkb} hold and let the initial condition $\hat{W}_1^\eps\in~\R^2$ be $\eps$--uniformly bounded above and below. Then $\exists\,\eps_1\in(0,\eps_0]$ such that the WKB-marching method for \eqref{ivp:tp_mod2_W} yields a sequence of vectors $\hat{W}_n^\eps\in~\R^2$, $n=1,\ldots,N$ that is uniformly bounded from above and below, i.e.\ 
\begin{equation}\label{eqn:ul_unif_bnd_Wn}
C_5\leq\|\hat{W}_n^\eps\|\leq C_6\,;\quad n=1,\ldots,N\;,
\end{equation} 
where the constants $C_5,C_6>0$ are independent of $0<\eps\leq\eps_1$ and of the numerical grid on $[x_1,1]$.
\begin{proof}
A proof can be carried out exactly as in \cite[Lemma 3.5]{AN18}, with the only difference that the initial condition $\hat{W}_1$ is now $\eps$--dependent (but uniformly bounded above and below).
\end{proof}
\end{lemma}

In particular Lemma \ref{lem:unif_bnd_Wn} applies to the $\eps$--dependent initial condition $\hat{W}_1:=\hat{W}(x_1)$ obtained from \eqref{ivp:tp_mod2} via the transformation \eqref{eqn:trafo_system_W}. $\hat{W}_1=\Theta_\eps(1)$ because of \eqref{eqn:unif_bnd_ic_ivp} and the $\eps$--uniform equivalence of the norms $\|\hat{W}(x_1)\|$ and $\|(\hat{\psi}(x_1),\eps\hat{\psi}^\prime(x_1))^\top\|$.


\subsection{Error estimates including the turning point}\label{subsec:errest_tp}

The following result is a simple consequence of Theorem 3.1 in \cite{ABN11}. It shows that the WKB-marching method applied to the IVP \eqref{ivp:tp_mod2} ---however, with a different IC compared to \cite{ABN11}--- yields the same $h$-- and $\eps$--order as when applied to the IVP proposed in \cite{ABN11}. So we obtain:

\begin{proposition}[see Thm. 3.1 in \cite{ABN11}]\label{prop:wkbmm2}
Let Assumptions \ref{ass:coeff_a} and \ref{ass:a_eps_wkb} on the coefficient function $a(x)$ be satisfied. Then the global error of the second order WKB-marching method for the IVP \eqref{ivp:tp_mod2} satisfies
\begin{equation}\label{eqn:wkbmm2_errest_W}
\|\hat{W}(x_n)- \hat{W}_n\|\leq C\tfrac{h^\gamma}{\eps} + C\eps^3 h^2 \,;\quad 1\leq n \leq N\,;\quad\forall\eps\in(0,\eps_1]\,,
\end{equation}
with a constant $C$ independent of $n$, $h$ and $\eps$. Here, $\gamma>0$ is the order of the chosen numerical integration method for evaluating the phase integral $\phi$ from \eqref{eqn:phase_beta}.
\end{proposition}

\begin{rem}\label{rem5.3}
\todo{
	In the error analysis of \cite{ABN11}, as well as in Proposition \ref{prop:wkbmm2} above, the possible error of the matrices $A_n^1, A_n^2$ in \eqref{eqn:wkb_iter} arising from an incorrect phase $\phi$ is not taken into account. Errors of $\phi$ are only considered for the back transformation \eqref{eqn:wkb_back_trafo}. However, in the recent, more complete error analysis in \cite{AKU18}, both occurrences of the error of $\phi$ are included. It would lead to a third error term in \eqref{eqn:wkbmm2_errest_W} that is $\calO(\eps^2)$ and includes the $W^{2,\infty}$--error of the phase (see \cite[Theorem 3.2]{AKU18}). We omit this additional error term here for brevity of the presentation.
}
\end{rem}

\begin{proof}[Proof of Proposition \ref{prop:wkbmm2}]\label{prf:wkbmm2}
	The proof is using the main result \cite[Theorem 3.1]{ABN11} for the IVP 
	\begin{equation} \label{ivp:aban}
	\begin{cases}
	\eps^2\varphi^{\prime\prime}(x) + a(x) \varphi(x) = 0\;, &\quad x\in[x_1,1]\;,\\
	\varphi(x_1) = 1\;, \\
	\eps\varphi^\prime(x_1) = -\im\sqrt{a(x_1)}\;,
	\end{cases}
	\end{equation}
	and it only remains to generalize it to the initial condition in \eqref{ivp:tp_mod2}.
	Let $Y(x) = (y_1(x), y_2(x))^\top$ be the vector valued solution to \eqref{ivp:aban} after the transformation via \eqref{eqn:trafo_system_W}, and $Y_n$ is its numerical approximation at $x_n$ obtained via the WKB-marching method. Due to Theorem 3.1 of \cite{ABN11} it holds 
	\begin{equation*}
	\|Y(x_n)- Y_n\|\leq C\tfrac{h^\gamma}{\eps} + C\eps^3 h^2 \,.
	\end{equation*}
	
	Next we give a transformation formula to connect $\varphi$ with $\hat{\psi}$, the solution of \eqref{ivp:tp_mod2}. One easily verifies that
	\begin{equation}\label{eqn:trafo_phi_psi}
	\hat\psi(x) = \Re\big(\rho_1(\eps) \varphi(x)\big)\;,
	\end{equation}
	with the $\eps$--dependent constant
	\begin{equation*}
	\rho_1(\eps) := \eps^{-\frac{1}{6}}\Ai(-\tfrac{x_1}{\eps^{2/3}}) - \im\eps^{\frac{1}{6}}x_1^{-\frac{1}{2}}{\Ai}^\prime(-\tfrac{x_1}{\eps^{2/3}})\;.
	\end{equation*}
	Using the asymptotic representations \eqref{eqn:asy_rep_airy} we verify 
	\begin{equation*}
	\rho_1(\eps) = \frac{1}{x_1^{1/4}\sqrt{\pi}}\left[ \cos(\xi(x_1))-\im\sin(\xi(x_1)	) + \calO(\eps)\right] =  \Theta_\eps(1)\;, \quad \eps\to 0\;,
	\end{equation*}
	where $\xi(x):= \tfrac{2x^\frac{3}{2}}{3\eps}- \tfrac{\pi}{4}$. Thus it holds 
	\begin{eqnarray}\nonumber
	\|\hat{W}(x_n) -\hat{W}_n\| &=& \|\Re\big(\rho_1(\eps)\left[Y(x_n)-Y_n\right]\big)\|\\\nonumber
	&\leq &\left|\rho_1(\eps)\right|\| Y(x_n)-Y_n\|\\\nonumber
	&\leq & C\tfrac{h^\gamma}{\eps} + C\eps^3 h^2\;,
	\end{eqnarray}
	with $C>0$ independent of $n$, $h$ and $\eps$.
\end{proof}
\todo{
In some applications the phase $\phi$ is exactly computable: In quantum tunneling models, e.g., the crystalline heterostructure leads to a piecewise linear potential $V(x)$ (with jumps due to the contact potential difference), and hence piecewise linear $a(x)$.} In this case the $\frac{h^\gamma}{\eps}$--error term in \eqref{eqn:wkbmm2_errest_W} drops out and the scheme satisfies an $\eps$--uniform, second order in $h$ error estimate. Moreover it is even asymptotically correct with respect to $\eps$. The opposite situation, when $\phi(x)$ has to be computed numerically, will be discussed in Remark \ref{rem:phase_error} below.


\bigskip


\emph{\bf Scaling to fit the right boundary condition\/}\\
The numerical approximation $W_n$ for $n=1,\ldots,N$ to the solution vector $W(x_n)$ of \eqref{bvp:mod2} is obtained by first calculating the numerical approximation $\hat{W}_n$ for the IVP \eqref{ivp:tp_mod2_W} via the WKB-marching method. Then it is scaled with $\tilde{\alpha}:=\tilde{\alpha}(\hat{W}_N)$, i.e.\
\begin{equation*}
W_n := \tilde{\alpha}\,\hat{W}_n\;,\quad n = 1,\ldots,N\;,
\end{equation*} 
where $\hat{W}_N$ is the approximation to $\hat{W}(1)$ obtained in the last step of the WKB-scheme. Now we can give the error estimates for numerically solving the BVP \eqref{bvp:mod2} using the $\tilde{\alpha}$--scaled Airy function $\tilde{\alpha}\,\psi_-$ (with the choice $c_0:=\eps^{-\frac{1}{6}}$) on $[0,x_1]$ and the $\tilde{\alpha}$--scaled numerical solution $\tilde{\alpha}\,\hat\psi_{h,n}$ on $\{x_1,\ldots,x_N\}$, i.e.

\begin{equation}\label{eqn:psi_h2}
\psi_h(x):=
\begin{cases}\vspace{.4em}
\tilde{\alpha}\,\psi_{-}(x):=	\tilde{\alpha}\,\eps^{-\frac{1}{6}}\,\Ai\left(-\tfrac{x}{\eps^{2/3}}\right)\,,\quad &x\in[0,x_1]\;,\\
\tilde{\alpha}\,\hat\psi_{h,n}:= \tilde{\alpha}\,\hat{w}_n^1\,a(x_n)^{-\frac{1}{4}}\,,\quad &x\in\{x_1\ldots,x_N\}\;,
\end{cases}
\end{equation}
\begin{equation}\label{eqn:psiP_h2}
\eps\,\psi_h^\prime(x):=
\begin{cases}\vspace{.4em}
\tilde{\alpha}\,\psi_{-}^\prime(x): = -\tilde{\alpha}\,\eps^\frac{1}{6}\,{\Ai}^\prime\left(-\tfrac{x}{\eps^{2/3}}\right)\,,\quad &x\in[0,x_1]\;,\\
\tilde{\alpha}\,\eps\,\hat\psi_{h,n}^\prime:=\tilde{\alpha}\left[a^{1/4}(x_n)\,\hat{w}_n^2 - \tfrac{\eps\,a^\prime(x_n)}{4\,a^{5/4}(x_n)}\,\hat{w}_n^1\right]\,,\quad\!\! &x\in\{x_1\ldots,x_N\}\;.
\end{cases}
\end{equation}

We recall that the $\eps$--scaling of the Airy function on $[0,x_1]$ is important here to satisfy the IC at $x_1$ for the IVP \eqref{ivp:tp_mod2}. Next we give error estimates for the hybrid solution \eqref{eqn:psi_h2}, \eqref{eqn:psiP_h2}, i.e.\ Airy function on $[0,x_1]$ coupled to the WKB-solution on $[x_1,1]$. While our main strategy follows \S 3.5 of \cite{AN18}, the turning point at $x=0$, and thus, the unboundedness of $\psi_\eps(0)$, causes technical challenges. 

\begin{theorem}[Convergence of the Airy-WKB method]\label{thm:errest_global}
Let Assumptions \ref{ass:coeff_a} and \ref{ass:a_eps_wkb} be satisfied and $0<\eps\leq\eps_1$. Then the pair $(\psi_h,\eps\psi_h^\prime)$ satisfies the following error estimates:
\begin{enumerate}[label = \alph*)]
\item In the region $[0,x_1]$, we have 
\begin{equation}\label{eqn:errest_airy0x1}
\|e_h\|_{C[0,x_1]} \leq  C\tfrac{h^\gamma}{\eps^{7/6}} + C\eps^{\frac{17}{6}}h^2\;,\quad \eps\|e_h^\prime\|_{C[0,x_1]} \leq  C\tfrac{h^\gamma}{\eps} + C\eps^3 h^2\;,
\end{equation}
where $e_h(x) := \psi(x) - \psi_h(x)$.
\item In the region $[x_1,1]$, we have 
\begin{equation}\label{eqn:errest_wkbmm2x11}
|e_{h,n}|  + \eps|e_{h,n}^\prime| \leq  C\tfrac{h^\gamma}{\eps}+C\eps^3 h^2\;,\quad n=1,\ldots,N\;,
\end{equation}
where $e_{h,n}:= \psi(x_n)-\tilde{\alpha}\,\hat{\psi}_{h,n}$ and $e_{h,n}^\prime:= \psi^\prime(x_n)-\tilde{\alpha}\,\hat{\psi}_{h,n}^\prime$.
\item For the hybrid method on the interval $[0,1]$ we have the error estimate 
\begin{equation}\label{eqn:errest_01}
\|e_h\|_\infty  \leq C\tfrac{h^\gamma}{\eps^{7/6}} + C\eps^{\frac{17}{6}}h^2 \;,\quad
\eps\|e_h^\prime\|_\infty  \leq C\tfrac{h^\gamma}{\eps}+C\eps^3 h^2 \;,
\end{equation}
where $\|e_h\|_\infty := \max\{\|e_h\|_{C[0,x_1]}; \max\limits_{n=1,\ldots,N}|e_{h,n}|\}$.
\end{enumerate}
\end{theorem}

\begin{rem}\label{rem:phase_error}
 The $\frac{h^\gamma}{\eps}$ (or $\frac{h^\gamma}{\eps^{7/6}}$) term drops out if the phase $\phi$ in \eqref{eqn:phase_beta} is explicitly integrable, leading to a second order scheme (in $h$) that is asymptotically correct with respect to $\eps$. But when $\phi$ has to be computed numerically, e.g., via Simpson's rule where $\gamma=4$, the scheme is still second order in $h$ as long as $h$ is bounded by $\calO(\eps^{\frac{7}{12}})$. And for the (slightly) less restrictive step size bound $h=\calO(\sqrt{\eps})$, the order of the scheme reduces to $h^{\frac{5}{3}}$. As a comparison, we note that the bound $h=\calO(\sqrt{\eps})$ is well known for the WKB approximation of highly oscillatory problems without a turning point (see \cite{LJL05,ABN11}), yielding a scheme of order $h^2$. 
Using a spectral method for the phase integral allows to drastically reduce the quadrature error for $\phi$, as illustrated in \cite{AKU18}.
\end{rem}

\begin{proof}{[of Theorem \ref{thm:errest_global}]}
Within this proof, we will use Lemma~\ref{lem:alpha_lip_bnd} multiple times for arguments $\hat{W}(x_n)$ as well as $\hat{W}_n$. Thus we choose $\delta:=\min(C_3,C_5)$ with the lower bounds $C_3$ and $C_5$ on the arguments obtained in Lemma~\ref{lem:unif_bnd_W} and Lemma \ref{lem:unif_bnd_Wn}. It is crucial here that the solution vector $\hat{W}(x_n)$ of the IVP \eqref{ivp:tp_mod2_W} as well as the numerical approximation $\hat{W}_n$ are in $\R^2$ (see Lemma \ref{lem:unif_bnd_Wn}). Then the map $\tilde{\alpha}$ is Lipschitz continuous with a constant $L_{\tilde{\alpha}}$ and uniformly bounded by a constant $C_{\tilde{\alpha}}$, both independent of $0<\eps\leq\eps_0$, as stated in Lemma \ref{lem:alpha_lip_bnd}.

\begin{enumerate}[label=\alph*)]
\item{
For $x\in[0,x_1]$ we first recall that 
\begin{equation*}
\psi(x) = \alpha(\hat{\psi}(1),\hat{\psi}^\prime(1))\,\psi_{-}(x) = \tilde{\alpha}(\hat{W}(1))\,\psi_{-}(x)\,,
\end{equation*}
cf.\ \eqref{eqn:psi_scale}, \eqref{eqn:alpha_scal_W}. Hence, with \eqref{eqn:psi_h2} we have
\begin{align}\nonumber
|\psi(x)-\psi_h(x)| &= |\tilde{\alpha}(\hat{W}(1))-\tilde{\alpha}(\hat{W}_N)|\,|\psi_-(x)|\\\nonumber
&\leq L_{\tilde{\alpha}}\,\|\hat{W}(1)-\hat{W}_N\|\,\eps^{-\frac{1}{6}}\,\left|\Ai\left(-\tfrac{x}{\eps^{2/3}}\right)\right| \\\nonumber
&\leq C\,\left(\tfrac{h^\gamma}{\eps}+\eps^3 h^2\right)\,\eps^{-\frac{1}{6}}\,\left|\Ai\left(-\tfrac{x}{\eps^{2/3}}\right)\right|\\\nonumber
&\leq C\,\left(\tfrac{h^\gamma}{\eps^{7/6}}+\eps^{\frac{17}{6}} h^2\right)\,,
\end{align}
where we used the $\eps$--uniform boundedness of the (scaled) Airy function $\Ai\left(-\tfrac{x}{\eps^{2/3}}\right)$ on $\R^+$. We also note that the term $\eps^{-\frac{1}{6}}$ cannot be compensated by $\Ai\left(-\tfrac{x}{\eps^{2/3}}\right)$, since the latter term takes a constant value (independent of~$\eps$) at $x=0$. We also used the Lipschitz continuity of $\tilde{\alpha}$, in addition to Proposition \ref{prop:wkbmm2}.

For the derivative we estimate as follows:
\begin{align}\nonumber
\eps|\psi^\prime(x)-\psi_h^\prime(x)| &= \eps|\tilde{\alpha}(\hat{W}(1))-\tilde{\alpha}(\hat{W}_N)|\,|\psi_-^\prime(x)|\\\nonumber
&\leq L_{\tilde{\alpha}}\,\|\hat{W}(1)-\hat{W}_N\|\,\left|\eps^{\frac{1}{6}}{\Ai}^\prime\left(-\tfrac{x}{\eps^{2/3}}\right)\right|\\\nonumber
&\leq C\,\left(\tfrac{h^\gamma}{\eps}+\eps^3 h^2\right)\,,
\end{align}
where we used that $|\eps^{\frac{1}{6}}{\Ai}^\prime\left(-\tfrac{x}{\eps^{2/3}}\right)|\leq c$. This can be argued in the same manner as in Step 5 of the proof of Proposition \ref{prop:bnd_mod2}.
}
\item{It is convenient to use the vector notation $W$ on $[x_1,1]$. Note that scaling $\hat{W}$ with the constant $\tilde{\alpha}$ is equivalent to scaling $(\hat{\psi},\eps\hat{\psi}^\prime)$. Therefore the estimates after the $\tilde{\alpha}$--scaling are
\begin{align}\nonumber
\|W(x_n) - W_n\| &= \|\tilde{\alpha}(\hat{W}(1))\,\hat{W}(x_n) - \tilde{\alpha}(\hat{W}_N)\,\hat{W}_n\|\\\nonumber
&\leq   |\tilde{\alpha}(\hat{W}(1)) - \tilde{\alpha}(\hat{W}_N)|\,\|\hat{W}(x_n)\| + |\tilde{\alpha}(\hat{W}_N)|\,\|\hat{W}(x_n) - \hat{W}_n\|\\\nonumber
&\leq  L_{\tilde{\alpha}} \|\hat{W}(1) - \hat{W}_N\| \, C_4 + C_{\tilde{\alpha}}\,\|\hat{W}(x_n) - \hat{W}_n\|\\\nonumber
&\leq   C\,\left(\tfrac{h^\gamma}{\eps}+\eps^3 h^2\right)\,, \quad n=1,\ldots,N\,.
\end{align}
In the second to last line we used the Lipschitz continuity and boundedness of $\tilde{\alpha}$ as well as \eqref{eqn:ul_unif_bnd_W}. In the last line we used the estimate \eqref{eqn:wkbmm2_errest_W} twice.
Due to the ($\eps$--uniform) equivalence of $\|W\|$ and $\|(\psi,\eps\psi')\|$, the estimate above yields the desired bound on the interval $[x_1,1]$. 
}
\item{The overall estimate on $[0,1]$ is a combination of the previous two. 
}\end{enumerate}
\end{proof}


\subsection{Numerical results}\label{subsec:num_res}

In this subsection we will present numerical results to illustrate the error estimates of Theorem \ref{thm:errest_global} for the Airy-WKB method. 


\begin{figure}[!ht]
	\includegraphics[trim=.8cm .5cm 1.2cm 1cm,scale=0.42]{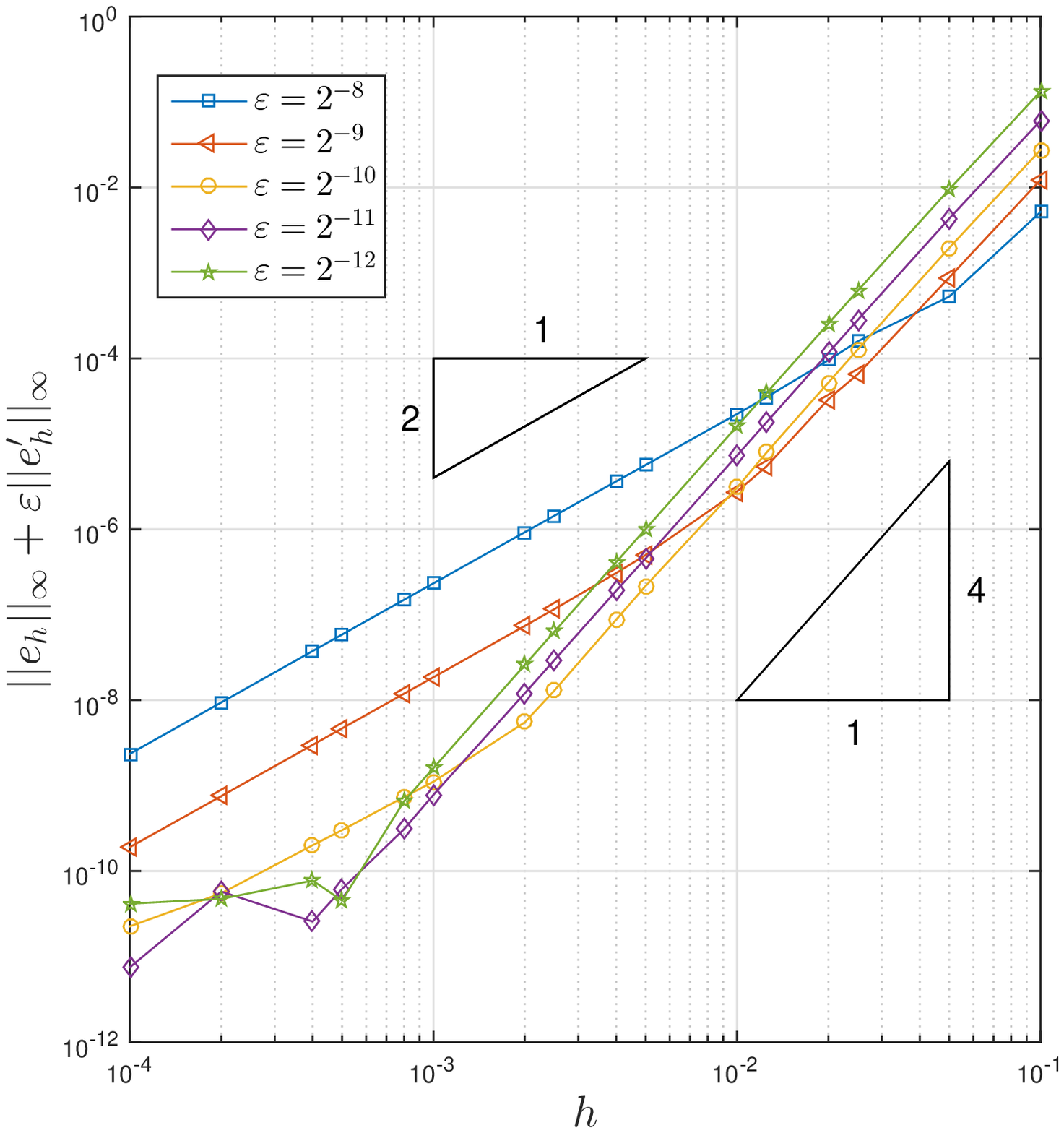}
	\hfil
	\includegraphics[trim=.8cm .5cm 1.2cm 1cm,scale=0.42]{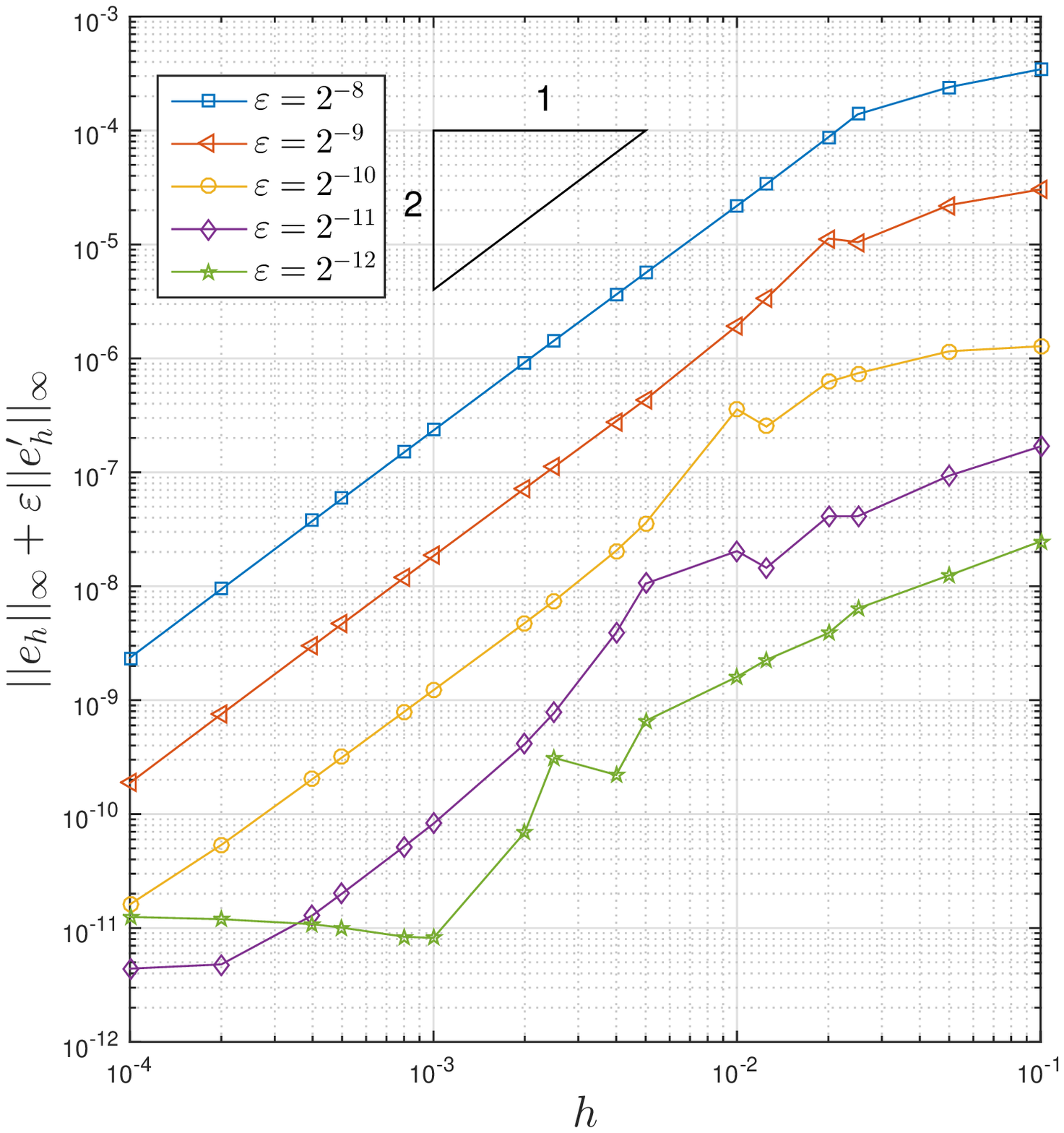}
	\caption{Absolute error on a log-log scale for the Airy-WKB method on $[0,1]$ with $x_1=0.1$ and the linear potential $a(x)=x$ for several values of $\eps$. $h$ is the step size for the WKB-marching method. \todo{On the left are the results with the phase $\phi$ computed numerically via the composite Simpson rule, and on the right the results using the explicitly known phase.}}
	\label{fig:airy-wkb-err}
\end{figure}

Fig.\ \ref{fig:airy-wkb-err} shows the error of the Airy-WKB method on $[0,1]$ with coefficient function $a(x)=x$ and ``switching-point'' $x_1=0.1$. This is a convenient test case since its solution is explicitly known (an Airy function), yet the numerical WKB method on $[x_1,1]$ is not trivial. Moreover, the integral of the phase $\phi(x)$, needed for the WKB-marching method, is explicitly available without numerical integration. \todo{For comparison we shall present two simulations, one with the exact phase $\phi(x)$ and one with a numerically computed phase (used both in \eqref{eqn:wkb_iter} and \eqref{eqn:wkb_back_trafo}).  In the right plot, the $\frac{h^\gamma}{\eps^{7/6}}$--term drops out of the error estimate  \eqref{eqn:errest_01},} yielding an asymptotically correct scheme as $\eps\to 0$ (for fixed step size $h$). 

\todo{In the left plot, the first term of \eqref{eqn:errest_01} (i.e. $\frac{h^\gamma}{\eps^{7/6}}$, originating from the numerical integration of the phase $\phi$) is dominant for large values of $h$ and/or small values of $\eps$. Hence, the error behaves like $\frac{h^4}{\eps^{7/6}}$, due to the Simpson rule with $\gamma=4$, as visualized by the right slope triangle. At $h=0.1$ we can clearly see an inversion of the error curves ($\eps=2^{-12}$ at the top and $\eps=2^{-8}$ at the bottom) with an $\eps$--dependence of almost exactly $\mathcal O(\eps^{-7/6})$. This error term originating from the numerical integration of the phase could be reduced to machine precision by using the spectral method proposed in \cite{AKU18}.}

\todo{In the right plot (and likewise in the left plot for small $h$ and large $\eps$)} we clearly observe the quadratic convergence rate in $h$ for each value of $\eps$. The error in $\eps$ is decreasing with order of about $\eps^{3.4}$ to $\eps^{3.6}$ and therefore better than the predicted estimates of order $\eps^{17/6}$ \todo{from the second term in \eqref{eqn:errest_01}}. This improved $\eps$--order of the error does not originate in the choice of a \emph{linear\/} potential, as a similar observation was already made in the error plot for the WKB-marching method in \cite[Fig.\ 3.1 (right)]{ABN11}. There, a simulation with a \emph{quadratic\/} coefficient function was chosen and the error order in $\eps$ was showing better results than the predicted order of $\eps^3$. 

\todo{For small values of both $h$ and $\eps$, the error is very small, such that it gets eventually polluted by round-off errors (due to double precision computations in Matlab). Hence, in this case the error is not showing a simple dependence on $h$ and/or $\eps$, and is around the order of $10^{-11}$.}

\begin{figure}[!ht]
	\includegraphics[scale=0.8]{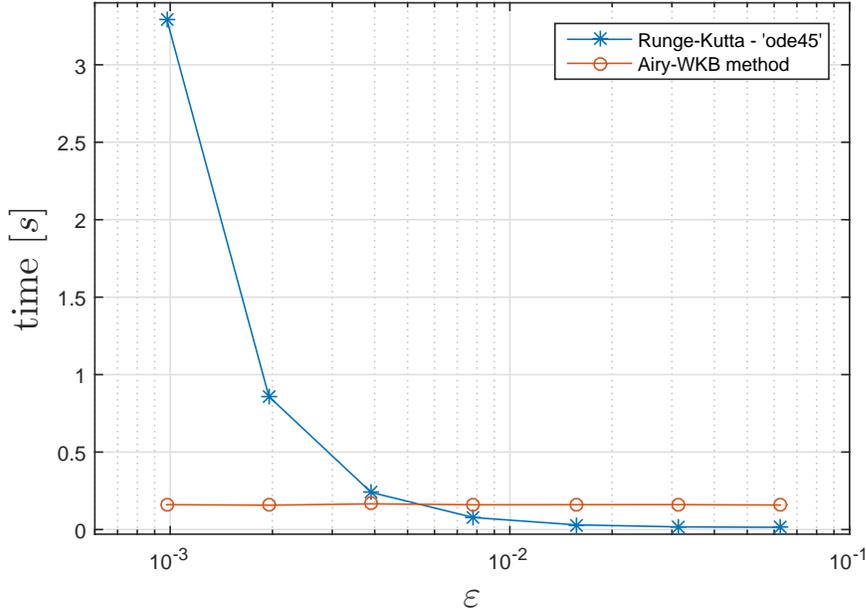}
	\caption{Run times of the Airy-WKB method in comparison to the standard Runge-Kutta Matlab solver `ode45' plotted against $\eps$ in a semi-log plot. We used the coefficient function $a(x)=x$, and $\eps = 2^{-4},\ldots,2^{-10}$. For each $\eps$, the prescribed error tolerance of `ode45' is `fitted' for the two methods to be comparable.}
	\label{fig:airy-wkb-rkode45-comp}
\end{figure}

In order to illustrate the efficiency of the Airy-WKB method, we make a comparison to a standard Runge-Kutta method, the `ode45' single step Matlab solver. This solver is based on an explicit Runge-Kutta $(4,5)$ formula (the Dormand-Prince pair). It is adaptive to attempt to optimize the step size $h$ using error estimates obtained from the comparison of a $4^{th}$ and $5^{th}$ order approximation. In this example we used (again) the linear coefficient function $a(x)=x$, such that we have the explicit formula for the exact solution at hand, in terms of the Airy function $\Ai$. Additionally, the phase is explicitly available without numerical integration, and thus it does not contribute to an additional error or increase of the run time. 
Also, in this case the WKB scheme is asymptotically correct.


The goal of this test is to compare (for several fixed values of $\eps$) the run times of the Airy-WKB and the Runge-Kutta methods, when both methods achieve (almost) the same numerical accuracy. To match the accuracies, we first determined the error of the Airy-WKB scheme, in the norm ${\|e_h\|_\infty+\eps\|e_h^\prime\|_\infty}$ (as defined in Theorem \ref{thm:errest_global}) for a grid with uniform step size $h=10^{-3}$. Then we specify an error tolerance for `ode45' (somewhat by trial and error) to obtain a matching approximation error. Fig.\ \ref{fig:airy-wkb-rkode45-comp} then gives a comparison of the run times: As expected, the Airy-WKB method's run times stay constant for decreasing $\eps$, since we leave the step size $h$ fixed. But the Runge-Kutta method's run times grow strongly for smaller $\eps$. This is a consequence of the following two contributing parts. One is due to the prescribed error tolerances. These need to go down, since the Airy-WKB method has decreasing error as $\eps\to0$ (even for a fixed step size), and we need to obtain comparable errors for both methods. The second reason is that, for smaller $\eps$, the oscillatory solution to the problem exhibits higher and higher frequencies, i.e. of order $\calO(\frac{\sqrt{a}}{\eps})$. Therefore the Runge-Kutta method needs to refine the step size $h$ even more to resolve the oscillations and to meet the prescribed error tolerances. 

\begin{table}
		\centering
		\captionof{table}{Run times and errors in comparison to a standard Runge-Kutta solver.}
		\label{tab:runtimes_error_awkb_rkode45} 
		\begin{tabular}{|c||c|c||c|c|}
			\hline
			& \multicolumn{2}{|c||}{Run time:} & \multicolumn{2}{|c|}{Error:} \\
			\hline
			$\eps$ & RK-ode45 & Airy-WKB & RK-ode45 & Airy-WKB \\\hline
			$2^{-4}$ & 1.4437e-02 & 1.5811e-01 & 7.7736e-05 & 7.0079e-05 \\			
			$2^{-5}$ & 1.7168e-02 & 1.6079e-01 & 6.0012e-05 & 6.5647e-05 \\
			$2^{-6}$ & 3.0221e-02 & 1.6042e-01 & 1.4315e-05 & 1.7369e-05 \\
			$2^{-7}$ & 7.7551e-02 & 1.5987e-01 & 2.1117e-06 & 2.4898e-06 \\
			$2^{-8}$ & 2.3965e-01 & 1.6568e-01 & 2.1115e-07 & 2.3355e-07 \\
			$2^{-9}$ & 8.5833e-01 & 1.5724e-01 & 1.8105e-08 & 1.8443e-08 \\
			$2^{-10}$ & 3.2963e+00 & 1.6063e-01 & 1.3234e-09 & 1.2300e-09 \\
			\hline
		\end{tabular}
\end{table}


\section{Generalization to quadratic potentials close to the turning point}\label{sec:mod3}

The goal of this section is a generalization of the hybrid method of Section \ref{sec:scat_mod}--\ref{sec:numeth_errana} to the situation when the potential $V(x)$ is \emph{quadratic\/} instead of linear in the vicinity of the turning point (which is still of first order at $x=0$). As we shall take a similar path as in the previous sections, not all of the (analogous) motivating deductions will be repeated. We start with specifying the assumptions on the coefficient function $a(x)$ analogously to Assumption \ref{ass:coeff_a}. 

\begin{ass}\label{ass:coeff_a2}\leavevmode
Let parts  \ref{item:cont_pot}, \ref{item:pot_1tp} and \ref{item:pot_rext} of Assumption \ref{ass:coeff_a} remain unchanged. 
\begin{enumerate}[label = \alph*')]
\setcounter{enumi}{2}
\item{More general than in Section \ref{sec:scat_mod}--\ref{sec:numeth_errana}, we now assume the potential to be quadratic in the left exterior \todo{and also in a (small) neighborhood of $x=0$.} More precisely we assume that $\exists\,x_1\in(0,1)$ such that $a(x) = k_1 x^2+k_2 x$ for $x\leq x_1$ with $k_1<0$ and $k_2>-k_1x_1>0$ such that the second zero of $a(x)$ is strictly larger than $x_1$, and thus, not included in $[0,x_1]$.}\label{item:pot_lext_quad_tp}
\end{enumerate}
\end{ass}

\begin{figure}[!ht]
	\includegraphics[trim=3.5cm 18.5cm 7.5cm 5cm,scale=1.25]{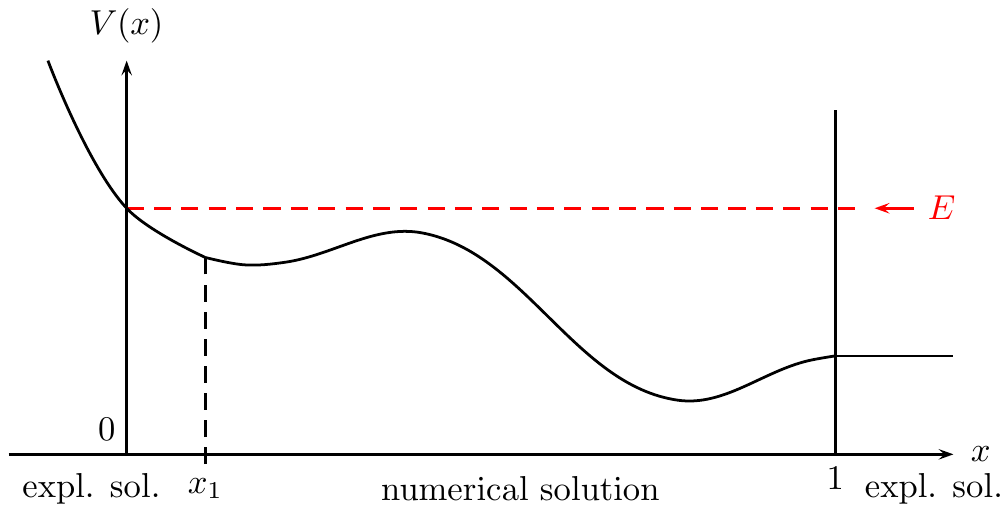}
	\caption{Sketch of the model described in Assumption \ref{ass:coeff_a2} with quadratic potential left of $x_1$. Electrons are injected from the right boundary $x=1$ and there is a turning point of first order at the left boundary $x=0$. The coefficient function is  $a(x) := E-V(x)$. The explicit solution form is available for $x\le x_1$ and for $x\ge1$; on $(x_1,1)$ the solution is obtained numerically.}
	\label{fig:potential_quad}
\end{figure}

Note that the above stated quadratic form of $a(x)$ is in its most general form as a turning point at $x=0$ requires $a(0)=0$.
We consider a potential $V(x)\to+\infty$ as $x\to -\infty$ (see Fig.\ \ref{fig:potential_quad}) which requires a scattering solution for the equation $\eps^2\psi''+a(x)\psi=0$ to decay for $x\to-\infty$. For a quadratic potential, this solution is given on $(-\infty,x_1]$ by the parabolic cylinder function (PCF) denoted by $U(\nu,z)$, cf.\ \cite[\S 12]{NHM10}. 

For $a(x) = k_1x^2+k_2x$, we have 
\begin{equation}\label{eqn:nu_zx}
\nu = -\frac{k_2^2}{8\eps\sqrt{-k_1^3}}<0 \;,\quad z(x) = \frac{k_2+2k_1x}{\sqrt{2\eps}(-k_1^3)^{1/4}}\in\R\,,
\end{equation}
and a fundamental set of solutions is $\{U(\nu,z(x)),U(-\nu,\im z(x))\}$. Here, $U(\nu,z(x))$ is the solution that stays bounded (and even decays) for $x\to-\infty$. This yields the following BVP with transparent BCs:

\begin{eqnarray}\label{bvp:mod3}
\begin{cases}
\eps^2 \psi^{\prime\prime}(x) +a(x)\psi(x) = 0\;, \quad x\in(x_1,1)\;,\\
\eps^\frac{1}{2}\psi^\prime(x_1)\,U\left(\nu,z(x_1)\right)+\sqrt{2}(-k_1)^\frac{1}{4}\,U^\prime\left(\nu,z(x_1)\right) \psi(x_1) = 0\;,\\
\eps\,\psi'(1) - \im\sqrt{a(1)}\psi(1) = -2\im\sqrt{a(1)}\;.
\end{cases}
\end{eqnarray}
Here and in the sequel the notation for the derivative of the parabolic cylinder function is $U^\prime(\nu,z)$ where the $^\prime$ always refers to the derivative w.r.t.\ the (second) argument $z$, not the order $\nu$.
An analogous result as Proposition \ref{prop:uniqueness} about existence and uniqueness of solutions holds for the BVP \eqref{bvp:mod3}.
As in Section  \ref{sec:ana_prob} this BVP will first be reformulated as the IVP
\begin{equation}\label{ivp:mod3}
\begin{cases}
\eps^2\hat\psi^{\prime\prime}(x) + a(x) \hat\psi(x) = 0\;, \quad x\in(x_1,1)\;,\\
\hat\psi(x_1) = c_1\,U\left(\nu,z(x_1)\right)\;, \\
\eps\,\hat\psi^\prime(x_1) = -c_1 \,\sqrt{2\eps}(-k_1)^\frac{1}{4}\,U^\prime\left(\nu,z(x_1)\right)\;,
\end{cases}
\end{equation}
with a constant $c_1\in\R\setminus\{0\}$ that we shall fix later. From here on we denote by $\hat{\psi}$ the solution to this IVP. The solution $\psi$ of \eqref{bvp:mod3} is then obtained by scaling $\hat{\psi}$ to fit the right BC of \eqref{bvp:mod3}, at $x=1$. Note that this scaling preserves the validity of the left BC of \eqref{bvp:mod3}, at $x_1$.


\subsection{Asymptotic blow-up at the turning point}\label{subsec:asy_tp}
In this section, we shall analyze the asymptotic behavior of solutions to the BVP \eqref{bvp:mod3} with quadratic potential close to the turning point. The following proposition will be used later to prove that the solution to \eqref{bvp:mod3} is not uniformly bounded w.r.t.\ $\eps\,$; this unboundedness arises at the turning point. 
\begin{proposition}[Asymptotics of the parabolic cylinder function]\label{prop:asy_pcf}
The function $U(\nu,z(x))$ with \eqref{eqn:nu_zx} has the following asymptotic representations for (a fixed) $x\in (0,-\frac{k_2}{k_1})$ as $\eps\to~0$:
\begin{equation}\label{eqn:asy_pcf_notp}
\begin{aligned}
U(\nu,z(x)) &\;=\; g(\mu)\left[2\left(\tfrac{1}{1-t^2}\right)^\frac{1}{4}\cos\left(\eta(t) - \tfrac{\pi}{4}\right) + \calO(\eps)\right]\,,\\[.5em]
\eps \tfrac{d}{dx}U(\nu,z(x)) &\;=\; -g(\mu)\left[\tfrac{k_2}{\sqrt{-k_1}}\left(1-t^2\right)^\frac{1}{4}\sin\left(\eta(t)-\tfrac{\pi}{4}\right) + \calO(\eps)\right]\,,
\end{aligned}
\end{equation}
with the notation
\begin{equation*}
t:=1+\tfrac{2k_1}{k_2}\,x\in(-1,1)\,, \quad \mu := \tfrac{k_2}{2(-k_1^3)^{1/4}}\eps^{-\frac{1}{2}}>0\,,
\end{equation*}
and
\begin{equation*}
\eta(t) := \tfrac{\mu^2}{2}\left(\arccos(t)-t\sqrt{1-t^2}\right)>0\,.
\end{equation*}
The ($\eps$--dependent) constant $g(\mu)$ is defined in \eqref{eqn:gmu} and \eqref{eqn:hmu}.
\end{proposition}
\noindent The lengthy proof is deferred to the Appendix \ref{app:prfs}.
\medskip

An essential requirement for the proofs in Section \ref{sec:ana_prob}--\ref{sec:numeth_errana} was the uniform boundedness of the initial condition $(\hat{\psi}(x_1),\eps\hat{\psi}^\prime(x_1))^\top$ w.r.t.\ $\eps$. Proposition \ref{prop:asy_pcf} shows that the initial condition in \eqref{ivp:mod3} has the asymptotic order $\calO(g(\mu))$ w.r.t.\ $\eps$, if $c_1$ is chosen independent of $\eps$. In order to obtain again $\eps$--uniform boundedness of the IC, we shall now choose $c_1$ depending on $\eps$. As $g(\mu)$ is defined as an asymptotic series and $g(\mu)=\Theta_\eps(h(\mu))$, we choose in the IC of \eqref{ivp:mod3}: 
\begin{equation*}
c_1=c_1(\eps):=\tfrac{1}{h(\mu)}=2^{\frac{1}{4}(\mu^{2}+1)}e^{\frac{1}{4}\mu^{2}}\mu^{\frac{1}{2}(1-\mu^{2})}\,,
\end{equation*}
with $h(\mu)$ defined in \eqref{eqn:hmu}. Using $\tfrac{1}{h(\mu)}$ is more practical than $\tfrac{1}{g(\mu)}$ since it can be implemented explicitly and it yields the same asymptotic scaling as $\tfrac{1}{g(\mu)}$. Then the IVP reads
\begin{equation}\label{ivp:tp_mod3}
\begin{cases}
\eps^2\hat\psi^{\prime\prime}(x) + a(x) \hat\psi(x) = 0\;, \qquad x\in(x_1,1)\;,\\
\hat\psi(x_1) = \tfrac{1}{h(\mu)} U\left(\nu,z(x_1)\right)\;, \\
\eps\hat\psi^\prime(x_1) = -\tfrac{1}{h(\mu)} \,\sqrt{2\eps}(-k_1)^\frac{1}{4}\,U^\prime\left(\nu,z(x_1)\right)\;.
\end{cases}
\end{equation}
After transforming $(\hat{\psi},\eps\hat{\psi}^\prime)$ to $\hat{W}\in\R^2 $ via the matrix $A(x)$ from \eqref{eqn:trafo_Ax} we get the vector-valued system
\begin{equation}\label{ivp:tp_mod3_W}
\begin{cases}
\hat{W}^\prime(x) = \left[\frac{1}{\eps}A_0(x)+\eps A_1(x) \right]\hat{W}(x)\;, &\quad x\in(x_1,1)\;,\\
\hat{W}(x_1)=\tfrac{1}{h(\mu)}A(x_1)\mtrx{c}{U\left(\nu,z(x_1)\right)\\ - \,\sqrt{2\eps}(-k_1)^\frac{1}{4}\,U^\prime\left(\nu,z(x_1)\right)}\in\R^2\;,
\end{cases}
\end{equation}
with the two matrices $A_0(x)$ and $A_1(x)$ as in \eqref{eqn:mtrx_a0a1}.
\medskip

A similar result as in Lemma \ref{lem:unif_bnd_W} yields the $\eps$--uniform boundedness of the analytic solution $(\hat{\psi},\eps\hat{\psi}^\prime)(x)$ to the IVP \eqref{ivp:tp_mod3} and of the vector valued solution $\hat{W}(x)$ to  \eqref{ivp:tp_mod3_W}:
\begin{lemma}\label{lem:unif_bnd_W_mod3}
Let $a(x)\in\mathcal{C}^2[x_1,1]$ and $a(x)\geq\tau_1>0$. Let $\hat\psi(x)$ be the solution to the IVP \eqref{ivp:tp_mod3}. Then $(\hat\psi(x),\eps\hat{\psi}^\prime(x))$ is uniformly bounded above and below, i.e.\ 
\begin{equation}\label{eqn:ul_unif_bnd_psi_mod3}
C_7\leq \left\|(\hat{\psi}(x),\eps\hat{\psi}^\prime(x))\right\|\leq C_8\;,\quad x\in[x_1,1]\;,
\end{equation} 
or equivalently 
\begin{equation}\label{eqn:ul_unif_bnd_W_mod3}
C_9\leq \|\hat{W}(x)\|\leq C_{10}\;,\quad x\in[x_1,1]\;,
\end{equation} 
where the constants $C_7,\ldots,C_{10}>0$ are independent of $0<\eps\leq\eps_0$.
\begin{proof}
The proof is identical to the proof of Lemma \ref{lem:unif_bnd_W}. It only remains to show the $\eps$--uniform boundedness above and below of the initial condition. As $x_1\in(0,1)$ we can use Proposition \ref{prop:asy_pcf} for an asymptotic representation of $(\hat\psi(x_1),\eps\hat{\psi}^\prime(x_1))$: With $t_1:= 1+\frac{2k_1}{k_2}x_1\in(-1,1)$ we get
\begin{equation*}
\left\|\mtrx{c}{\hat{\psi}(x_1)\\\eps\hat{\psi}^\prime(x_1)}\right\|^2 = \frac{4|f(\mu)|^2}{(1-t_1^2)^\frac{1}{2}}\bigg[|\cos(\xi(t_1))|^2+\tfrac{k_2^2(1-t_1^2)}{4(-k_1)}|\sin(\xi(t_1))|^2+\calO(\eps)\bigg]\,,
\end{equation*}
where $\xi(t):=\eta(t)-\tfrac{\pi}{4}$ and $f(\mu) := \frac{g(\mu)}{h(\mu)} = \Theta_\eps(1)$. As $\sin(\xi(t_1))$ and $\cos(\xi(t_1))$ are never simultaneously zero we get $\|(\hat{\psi},\eps\hat{\psi}^\prime)(x_1)\| = \Theta_\eps(1)$, proving uniform boundedness above and below. 
\end{proof}
\end{lemma}

Next we illustrate that the solutions of the BVP \eqref{bvp:mod3} become unbounded at the turning point $x=0$, analogously to Example \ref{exam:airy}. First we consider
\begin{exam}\label{exam:pcf}
Consider \eqref{bvp:mod3} with $x_1=0$ and $a(x) = x-\frac{x^2}{2}$ for $x\in[0,1]$ and $0<\eps<1$. Then the explicit solution reads
\begin{equation}\label{eqn:exampcf_bvp_mod3}
\psi_\eps(x) = \frac{2}{U(\nu,0)-\im\sqrt{\eps}\,2^{3/4}\,U^\prime(\nu,0)}U(\nu,z(x))\,,
\end{equation}
where 
\begin{equation*}
\nu = -\tfrac{1}{\sqrt{8}\eps}\,,\quad z(x) = \tfrac{2^\frac{1}{4}}{\sqrt{\eps}}(1-x)\,.
\end{equation*}

\begin{figure}[!ht]
	\includegraphics[scale=0.63]{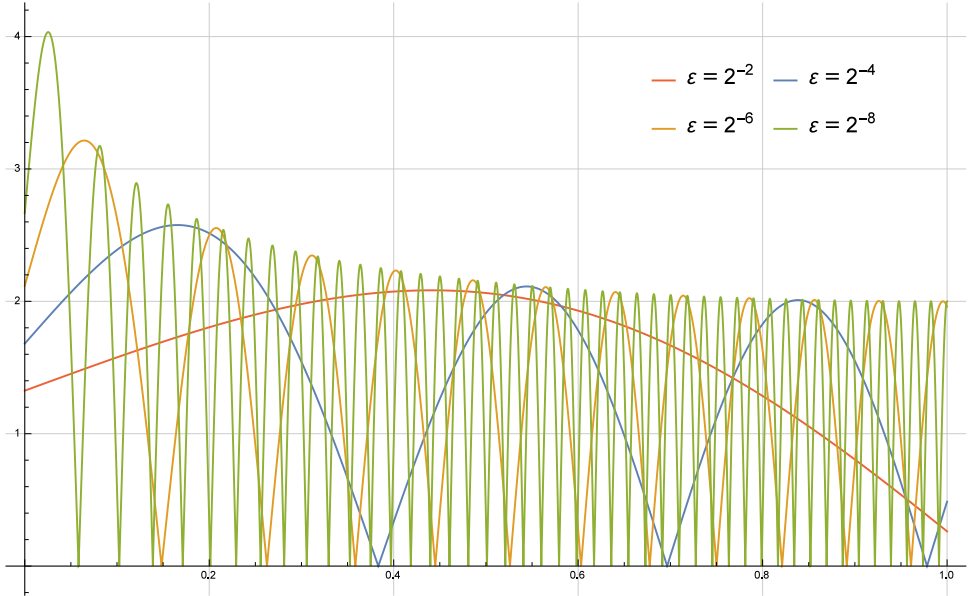}
	\caption{$|\psi_\eps(x)|$ for various values of $\eps$. At the turning point $(x=0)$ it increases with the order $\eps^{-\frac{1}{6}}$.}
	\label{fig:tp_non_epsunif_psi2}
\end{figure}
\begin{figure}[!ht]
	\includegraphics[scale=0.63]{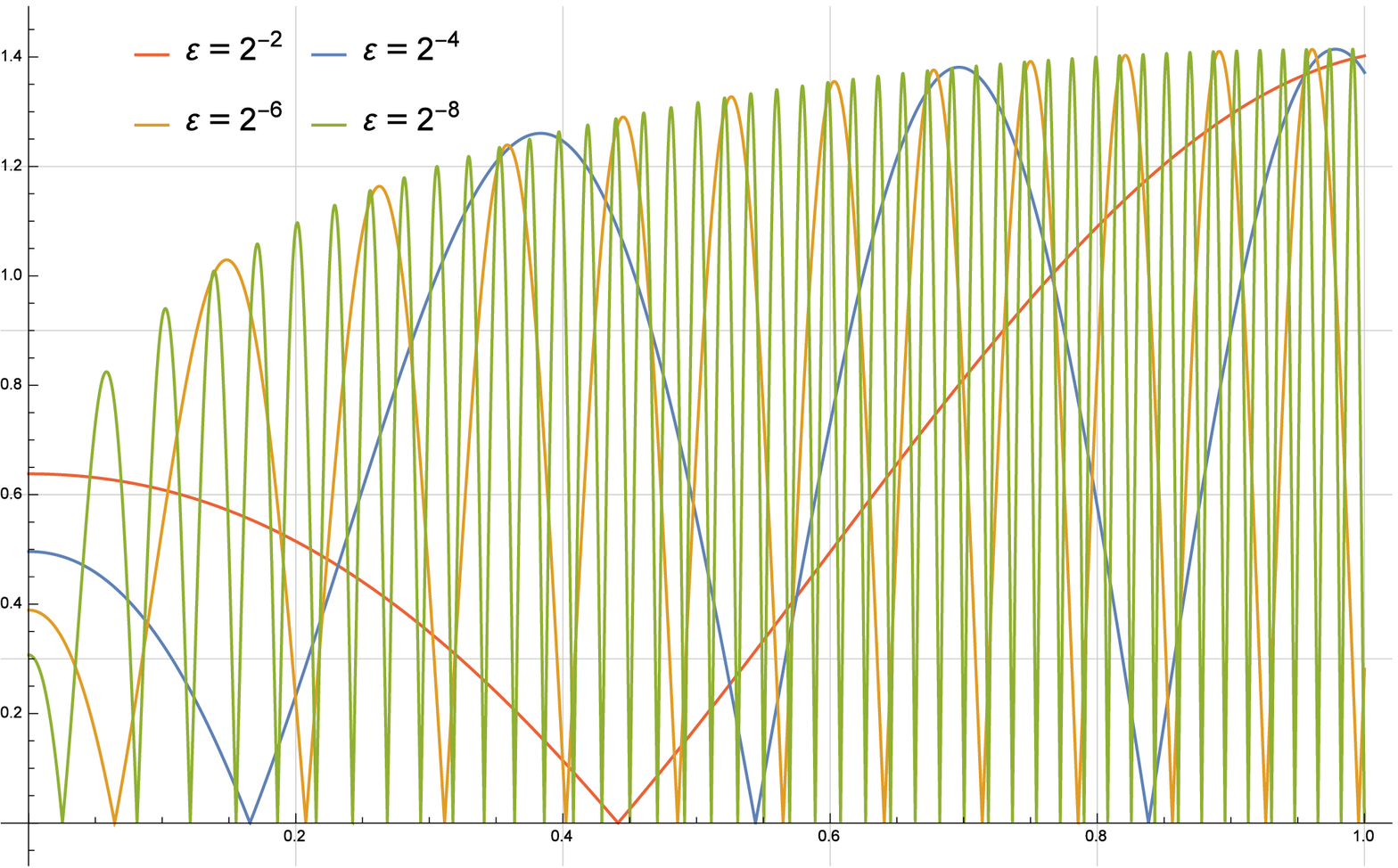}
	\caption{$\eps|\psi_\eps^\prime(x)|$ for various values of $\eps$.}
	\label{fig:tp_epsunif_psiPri2}
\end{figure}

\noindent The quadratic coefficient $a(x) = x-\tfrac{x^2}{2}$ has a turning point at $x=0$ and in Fig.\ \ref{fig:tp_non_epsunif_psi2} the absolute value of solutions $|\psi_\eps(x)|$ is plotted for various $\eps>0$: The solutions $|\psi_\eps(x)|$ are unbounded as $\eps\to 0$ at the turning point $x=0$. Fig.\ \ref{fig:tp_epsunif_psiPri2} shows the plot of the family $\{\eps|\psi_\eps^\prime(x)|\}$, which is bounded on $[0,1]$ uniformly in $\eps$. It is even decreasing (in $\eps\to 0$) at $x=0$.
\end{exam}

The generalization of Example \ref{exam:pcf} to potentials as in Assumption \ref{ass:coeff_a2} is the main result of the following proposition.

\begin{proposition}\label{prop:bnd_mod3}
Let $x_1\in(0,1)$ and $a(x)$ be as in Assumption \ref{ass:coeff_a2} and $\mathcal{C}^2$ on $[x_1,1]$. 
Then the family of solutions $\{\psi_\eps(x)\}$ to the BVP \eqref{bvp:mod3}, extended with the PCF solution on $[0,x_1]$, satisfies:
\begin{enumerate}[label = \alph*)]
\item{
$\|\psi_\eps\|_{L^\infty(0,1)}$ is of the (sharp) order $\Theta(\eps^{-\frac{1}{6}})$ for $\eps\to0$.
}\label{item:nonunifbnd_mod3_psi}
\item{
$\eps\|\psi_\eps^\prime\|_{L^\infty(0,1)}$ is uniformly bounded with respect to $\eps\to0$.
}\label{item:unifbnd_mod3_psiPri}
\end{enumerate}
\end{proposition}
The lengthy and involved proof is provided in Appendix \ref{prf:sol_order_tp}.


\subsection{Numerical method and error analysis}\label{subsec:meth_err}

In this section we extend the Airy-WKB method from Section \ref{sec:numeth_errana} to the more general case of a quadratic potential in the vicinity of the turning point. 
As the fundamental solution to the equation in \eqref{bvp:mod3} for quadratic $a(x)$ is a parabolic cylinder function (PCF), we will denote this method as PCF-WKB method.
The error estimates and convergence results will be essentially the same as for the Airy-WKB method in Section \ref{sec:numeth_errana}. 

For the notation we will again use Table \ref{tab:notation}, but instead of the BVP \eqref{bvp:mod2} and IVPs \eqref{ivp:tp_mod2}, resp.\ \eqref{ivp:tp_mod2_W} we consider the corresponding BVP \eqref{bvp:mod3} and IVPs \eqref{ivp:tp_mod3}, resp.\ \eqref{ivp:tp_mod3_W}.

With the uniformly bounded initial condition $\hat{W}(x_1)$ from \eqref{ivp:tp_mod3_W}, Lemma \ref{lem:unif_bnd_Wn} can be directly applied to the numerical approximation $\hat{W}_n^\eps$ obtained via the WKB-marching method. Hence, there exists $\eps_1\in(0,\eps_0]$ and $\eps$--independent constants $C_{11},C_{12}>0$ such that
\begin{equation}\label{eqn:unif_bnd_Bn3}
C_{11}\leq \left\|\hat{W}_n^\eps\right\|\leq C_{12}\,,\quad n=1,\ldots,N\,,\quad 0<\eps<\eps_1\,.
\end{equation}

Next we want to carry over the error estimates of the WKB-marching method \cite{ABN11}, just like in Proposition \ref{prop:wkbmm2}. Observe that, similar to \eqref{eqn:trafo_phi_psi}, one can write the solution $\hat{\psi}(x)$ to the IVP \eqref{ivp:tp_mod3} as
\begin{equation*}
\hat{\psi}(x) = \Re(\rho_2(\eps) \varphi(x))\,,
\end{equation*}
where $\varphi$ is the solution to \eqref{ivp:aban} and 
\begin{equation*}
\rho_2(\eps) := \tfrac{1}{h(\mu)\sqrt{a(x_1)}}\left[\sqrt{a(x_1)}U(\nu,z(x_1)) - \im\sqrt{2\eps}(-k_1)^\frac{1}{4}U^\prime(\nu,z(x_1)) \right]\,.
\end{equation*}
Using Proposition \ref{prop:asy_pcf} one can see that
\begin{equation*}
\rho_2(\eps) = \tfrac{\sqrt{2k_2}f(\mu)}{a(x_1)^\frac{1}{4}(-k_1)^\frac{1}{4}}\left[\cos(\xi(t_1)) - \im\sin(\xi(t_1)) + \calO(\eps)\right] = \Theta_\eps(1)\,,
\end{equation*}
and we get an analog error estimate as in Proposition \ref{prop:wkbmm2}: Under Assumptions \ref{ass:a_eps_wkb} and \ref{ass:coeff_a2}, and for $0<\eps\leq\eps_1$ it holds 
\begin{equation}\label{eqn:wkbmm2_errst_W3}
\|\hat{W}(x_n)- \hat{W}_n\|\leq C\tfrac{h^\gamma}{\eps} + C\eps^3 h^2 \,,\qquad 1\leq n \leq N\,.
\end{equation}
Here, $\hat{W}(x_n)$ is the exact solution to the IVP \eqref{ivp:tp_mod3_W} at $x_n$, and $\hat{W}_n$ is its numerical approximation obtained by the WKB-marching method. 

Using the numerical approximation $\hat{W}_n$ we shall denote the (numerical) approximation $\psi_h^{(\prime)}$ to the solution $\psi^{(\prime)}$ of the BVP \eqref{bvp:mod3} --extended to $[0,x_1]$-- as
\begin{equation}\label{eqn:psi_h3}
\psi_h(x):=
\begin{cases}\vspace{.4em}
\tilde{\alpha}\,\psi_{-}(x):=	\tilde{\alpha}\,\tfrac{1}{h(\mu)}\,U(\nu,z(x))\,,\quad &x\in[0,x_1]\;,\\
\tilde{\alpha}\,\hat\psi_{h,n}:= \tilde{\alpha}\,\hat{w}_n^1\,a(x_n)^{-\frac{1}{4}}\,,\quad &x\in\{x_1\ldots,x_N\}\;,
\end{cases}
\end{equation}
\begin{equation}\label{eqn:psiP_h3}
\eps\,\psi_h^\prime(x):=
\begin{cases}\vspace{.4em}
\tilde{\alpha}\,\psi_{-}^\prime(x): = -\tilde{\alpha}\,\tfrac{\sqrt{2\eps}}{h(\mu)}\,(-k_1)^\frac{1}{4}\,U^\prime(\nu,z(x))\,,\quad &x\in[0,x_1]\,,\\
\tilde{\alpha}\,\eps\,\hat\psi_{h,n}^\prime:=\tilde{\alpha}\,\left[a^{1/4}(x_n)\,\hat{w}_n^2 - \tfrac{\eps\,a^\prime(x_n)}{4\,a^{5/4}(x_n)}\,\hat{w}_n^1\right]\,,\quad\!\! &x\in\{x_1\ldots,x_N\}\,,
\end{cases}
\end{equation}
with the abbreviation $\tilde{\alpha}:=\tilde{\alpha}(\hat{W}_N)$.

With this notation we can now formulate the analog result to Theorem \ref{thm:errest_global} for quadratic potentials satisfying Assumption \ref{ass:coeff_a2}. The error orders are the same as in the case of a linear potential, since we considered a first order turning point at $x=0$ in both cases. 

\begin{theorem}[Convergence of the PCF-WKB method]\label{thm:errest_global_quad}
Let Assumptions \ref{ass:a_eps_wkb} and \ref{ass:coeff_a2} be satisfied and $0<\eps\leq\eps_1$. Then the pair $(\psi_h,\eps\psi_h^\prime)$ satisfies the same error estimates as in Theorem \ref{thm:errest_global}.
\begin{proof}
In this proof we are using the Lipschitz constant $L_{\tilde{\alpha}}$ and upper bound $C_{\tilde{\alpha}}$ for $\tilde{\alpha}$ from Lemma \ref{lem:alpha_lip_bnd} with $\delta := \min(C_9,C_{11})$ as lower bound for the arguments $\hat{W}(x_n)$ and $\hat{W}_n$, cf.\ \eqref{eqn:ul_unif_bnd_W_mod3} and \eqref{eqn:unif_bnd_Bn3}. 
\begin{enumerate}[label=\alph*)]
\item{For the error estimate on $\psi(x)$ with $x\in[0,x_1]$ we need the following estimate:
\begin{equation}\label{eqn:psi_min_bnd}
\begin{aligned}
|\psi_{-}&(x)| = \left|\tfrac{1}{h(\mu)}U(\nu,z(x))\right| = \left|\tfrac{g(\mu)}{h(\mu)}\left(2\sqrt{\pi}\varphi(t)\mu^\frac{1}{3}\Ai(\mu^\frac{4}{3}\zeta(t)) + \calO(\eps)\right)\right|\\
& \leq\tfrac{|f(\mu)|\,2^\frac{2}{3} \sqrt{\pi}\,k_2^\frac{1}{3}}{(-k_1)^\frac{1}{4}}\max_{t\in[1+\frac{2k_1}{k_2}x_1,1]}\left|\varphi(t)\Ai(\mu^\frac{4}{3}\zeta(t))\right|\eps^{-\frac{1}{6}} + \calO(\eps)\leq C_{\psi_-}\eps^{-\frac{1}{6}}\,,
\end{aligned}
\end{equation}
with an $\eps$--independent constant $C_{\psi_-}>0$. In the first line we used the asymptotic representation \eqref{eqn:asy_rep_pcf_tp} for $U(\nu,z(x))$, which is uniform in $x\in[0,x_1]$; the terms $\varphi(t)$ and $\zeta(t)$ are defined right after \eqref{eqn:asy_rep_pcf_tp}. In the last line we used that the $\eps$--dependent constant $f(\mu) = \Theta_\eps(1)$, and that the $\max$-term is $\eps$--uniformly bounded, see \eqref{eqn:max_airy_phi}. 
Using the estimate  \eqref{eqn:psi_min_bnd} yields 
\begin{align}\nonumber
|\psi(x)-\psi_h(x)| &= |\tilde{\alpha}(\hat{W}(1))-\tilde{\alpha}(\hat{W}_N)|\,|\psi_-(x)|\\\nonumber
&\leq L_{\tilde{\alpha}}\,\|\hat{W}(1)-\hat{W}_N\|\,C_{\psi_-}\eps^{-\frac{1}{6}}\\\nonumber
&\leq C\,\left(\tfrac{h^\gamma}{\eps}+\eps^3 h^2\right)\,\eps^{-\frac{1}{6}}= C\,\left(\tfrac{h^\gamma}{\eps^{7/6}}+\eps^{\frac{17}{6}} h^2\right)\,,
\end{align}
where we used \eqref{eqn:wkbmm2_errst_W3} in addition to the Lipschitz continuity of $\tilde{\alpha}$. 

For the $\eps$--scaled derivative $\eps\psi_{-}^\prime(x)$ in \eqref{eqn:psiP_h3} we have the asymptotic expansion \eqref{eqn:full_asy_pcf}. The term $\mu^{-\frac{4}{3}}\Ai(\mu^\frac{4}{3}\zeta(t))=\calO(\eps^\frac{4}{6})$, when including the turning point at $t=1$ (since $\zeta(1)=0$), and for all other $t\in[1+\frac{2k_1}{k_2}x_1,1)$ we have $\Ai(\mu^\frac{4}{3}\zeta(t)) = \calO(\eps^\frac{1}{6})$. Truncating the asymptotic expansion \eqref{eqn:full_asy_pcf} after the lowest order term in $\eps$ (which pertains to $D_0(\zeta)=1$)  shows that there exists an $\eps$--independent $C_{\psi_-^\prime}>0$, such that
\begin{equation*}
|\eps\psi_{-}^\prime(x)| = \bigg|-\frac{f(\mu)\,2^\frac{1}{3}\sqrt{\pi}k_2^\frac{2}{3}}{(-k_1)^\frac{1}{4}}\frac{{\Ai}^\prime(\mu^\frac{4}{3}\zeta(t))\eps^\frac{1}{6}}{\varphi(t)}+\calO(\eps^\frac{5}{6})\bigg|\leq C_{\psi_-^\prime}\,.
\end{equation*}
Here we used the fact that $|\frac{{\Ai}^\prime(\mu^{4/3}\zeta(t))\eps^{1/6}}{\varphi(t)}|$ for $t\in[1+\frac{2k_1}{k_2}x_1,1]$ is $\eps$--uniformly bounded above, see \eqref{eqn:max_ariyPri_phi}. Hence,
\begin{align}\nonumber
\eps|\psi^\prime(x)-\psi_h^\prime(x)| &= \eps|\tilde{\alpha}(\hat{W}(1))-\tilde{\alpha}(\hat{W}_N)|\,|\psi_-^\prime(x)|\\\nonumber
&\leq L_{\tilde{\alpha}}\,\|\hat{W}(1)-\hat{W}_N\|\,C_{\psi_-^\prime}\\\nonumber
&\leq C\,\left(\tfrac{h^\gamma}{\eps}+\eps^3 h^2\right)\,.
\end{align}
}
\end{enumerate}
The parts b) and c) are identical to the proof of the respective parts in Theorem \ref{thm:errest_global}.
\end{proof}
\end{theorem}

The $\tfrac{h^\gamma}{\eps^{7/6}}$ term from the numerical integration of the phase $\phi(x)$ from \eqref{eqn:phase_beta} appears again in the error estimates of Theorem \ref{thm:errest_global_quad}. As mentioned in Section \ref{subsec:errest_tp}, this term drops out for some applications, when the phase is explicitly integrable.


\subsection{Numerical results}\label{subsec:num_res_pcf}
The illustration of the error estimates of Theorem \ref{thm:errest_global_quad} for the PCF-WKB method is done analogously to Subsection \ref{subsec:num_res}. Here the coefficient function of \eqref{eqn:1d_schroed} is chosen as $a(x)= x-\frac{x^2}{2}$ on $[0,1]$. Therefore the solution $\psi$ is explicitly known as a parabolic cylinder function. The calculations are done in MATLAB\textsuperscript{\textregistered}, where the PCF is not implemented, but can be obtained via relations to other functions, i.e.\ the confluent hypergeometric function \cite[\S 12.7(iv)]{NHM10}. As the range of $\eps$ we chose $\eps = 2^{-5},\ldots,2^{-9}$, since for $\eps=2^{-10}$ the evaluation of the PCF returns \texttt{Inf} as values. This is because its order $\nu$ becomes large (negative) for small $\eps$, and thus, the PCF maps to very large values. For $\eps= 2^{-9}$ the evaluation already gets very inaccurate, such that we used \textit{Mathematica}\textsuperscript{\textregistered} for this case -- to evaluate the PCF for the reference solution. 

\todo{
For the plots in Fig.\ \ref{fig:pcf-wkb-err}, we computed the integral for the phase $\phi(x)$ numerically, once for the plot on the right with an error tolerance of $10^{-12}$ (such that the corresponding error term is negligible in comparison to the second error term in \eqref{eqn:errest_01}), and once using the composite Simpson rule for the phase in the plot on the left. We clearly see the $h^2$ convergence rate of the method for each $\eps$. An interesting difference can be seen in the left plot for small $\eps$ and large $h$: Here the $\frac{h^4}{\eps^{7/6}}$--term, originating from the numerical integration of the phase, is dominant.}

\todo{
In the region with quadratic convergence (i.e.\ when the second term in \eqref{eqn:errest_01} is dominant),} the error in $\eps$ is decreasing with order of about $\eps^{3}$ to $\eps^{3.7}$, and thus it is again (cf.\ Subsection \ref{subsec:num_res}) slightly superior to the predicted estimates of order $\eps^\frac{17}{6}$. 


\begin{figure}[!ht]
	\includegraphics[trim=.8cm .5cm 1.2cm 1cm,scale=0.42]{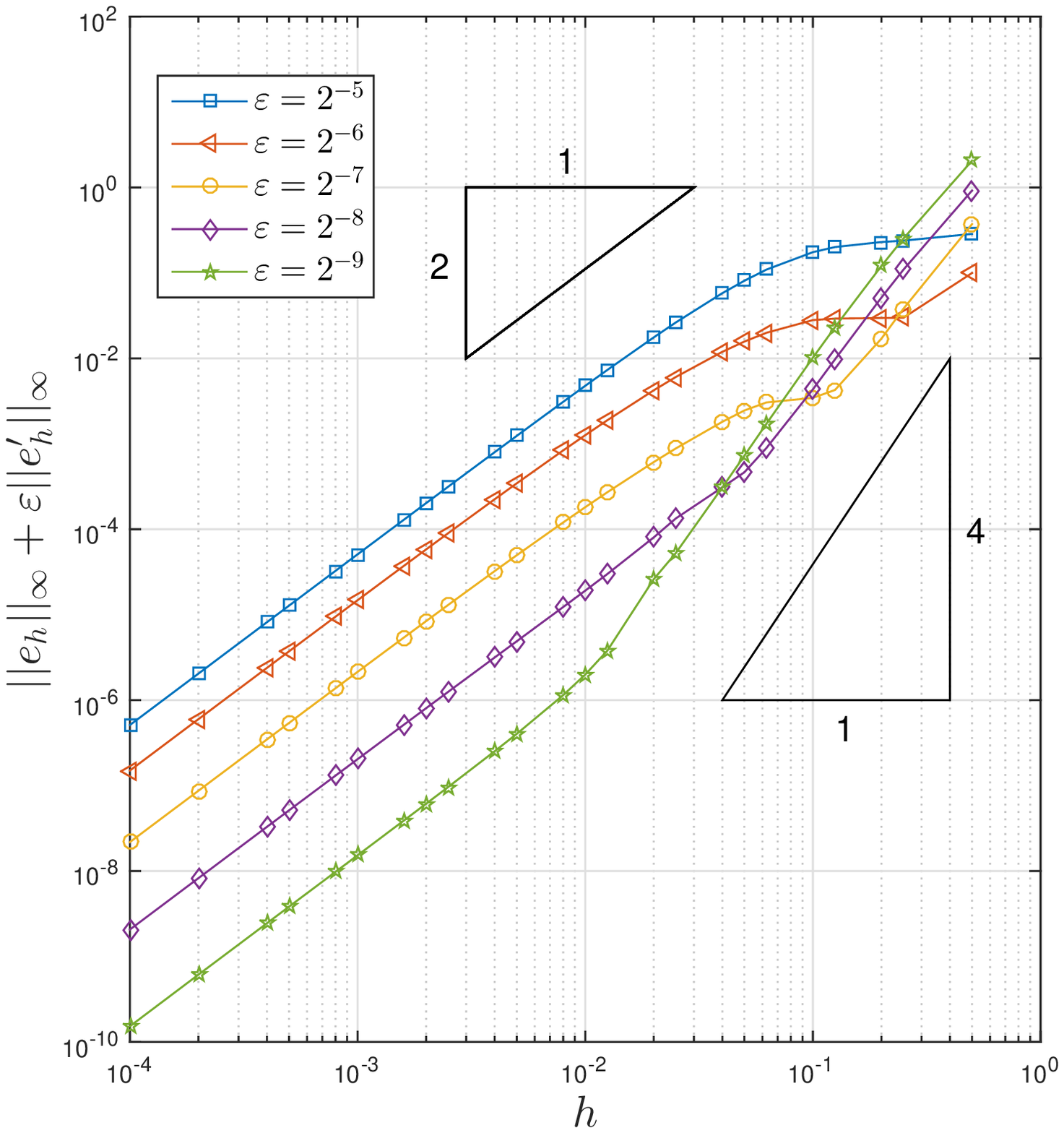}
	\hfil
	\includegraphics[trim=.8cm .5cm 1.2cm 1cm,scale=0.42]{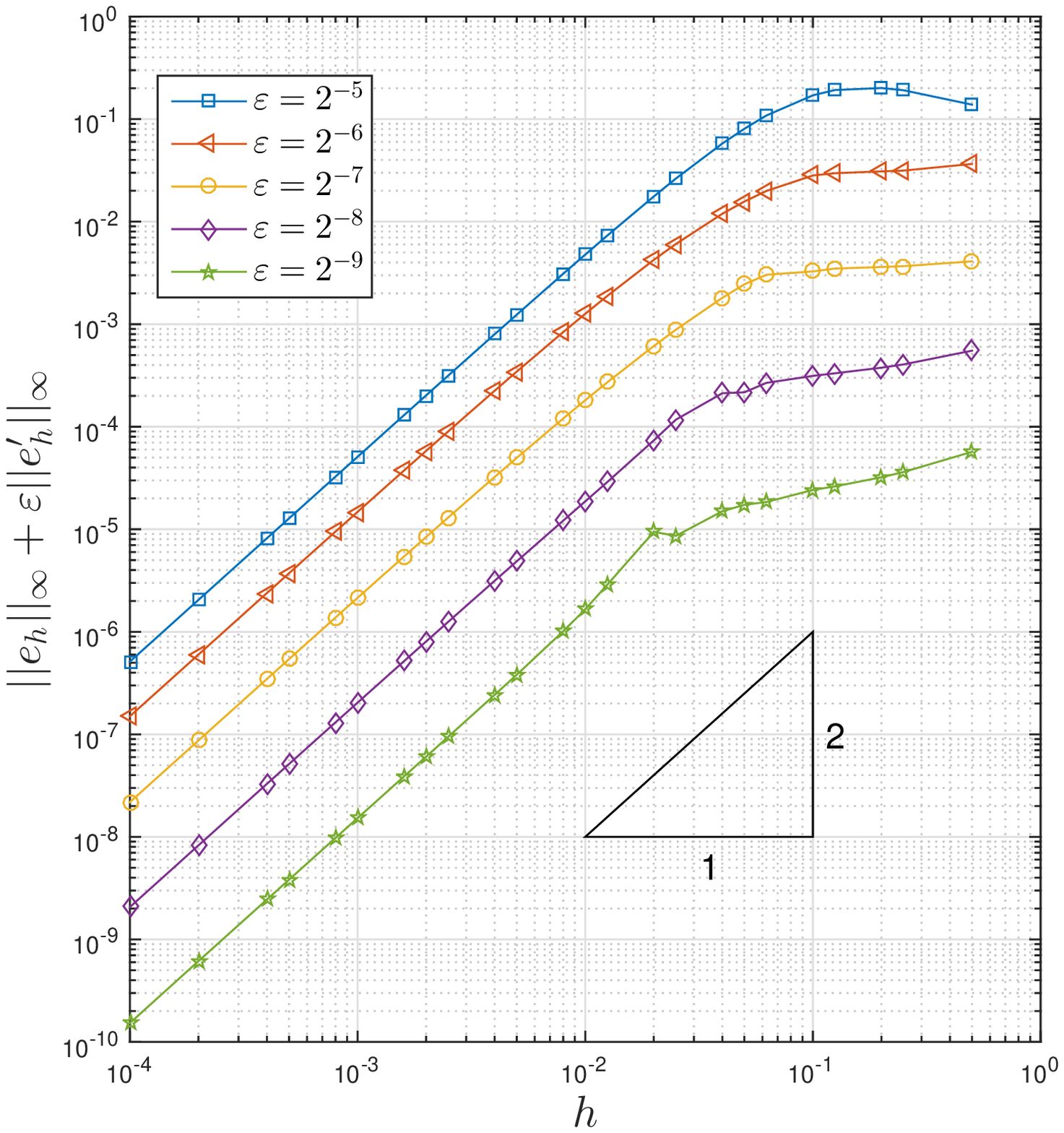}
	\caption{Absolute error on a log-log scale for the PCF-WKB method on $[0,1]$ with $x_1=0.1$ and the quadratic potential $a(x)=x-\frac{x^2}{2}$ for several values of $\eps$. $h$ is the step size for the WKB-marching method. \todo{On the left are the results with the phase $\phi$ computed numerically via the composite Simpson rule, and on the right the results using a numerically computed phase with high precission (error tolerance of $10^{-12}$).}}
	\label{fig:pcf-wkb-err}
\end{figure}


\section{Generalization and outlook}\label{sec7}

The next natural question is how to alleviate the restriction that the potential $V(x)$ should be (exactly) linear or quadratic in a neighborhood of the turning point. 
An obvious and frequently used strategy is to approximate the potential: In \cite{HH08} a piecewise constant approximation was used (actually in a regime without turning points), and in \cite[\S15.5]{Hal13}, \cite[\S7.3.1]{Nay73}, \cite[\S4.3]{Hol95} a linear approximation at the turning point was employed. But, as we shall illustrate next, such linear (or even higher order) approximations lead to errors that are unbounded as $\eps\to0$. Hence, they cannot serve as a starting point to construct uniformly accurate schemes. 
%
The error encountered by taking Airy functions (as solutions to the reduced problem --  i.e.\ the Airy equation) can only be uniformly bounded when confining to a region around the turning point that shrinks fast enough as $\eps\to 0$ (particularly like $o(\eps^{2/5})$, \cite{Mil06}). 
But, for the time being, the WKB-marching method requires a constant-in-$x$ transition point $x_1$.


\begin{figure}[!ht]
	\includegraphics[trim=3cm 21.5cm 5cm 2.3cm]{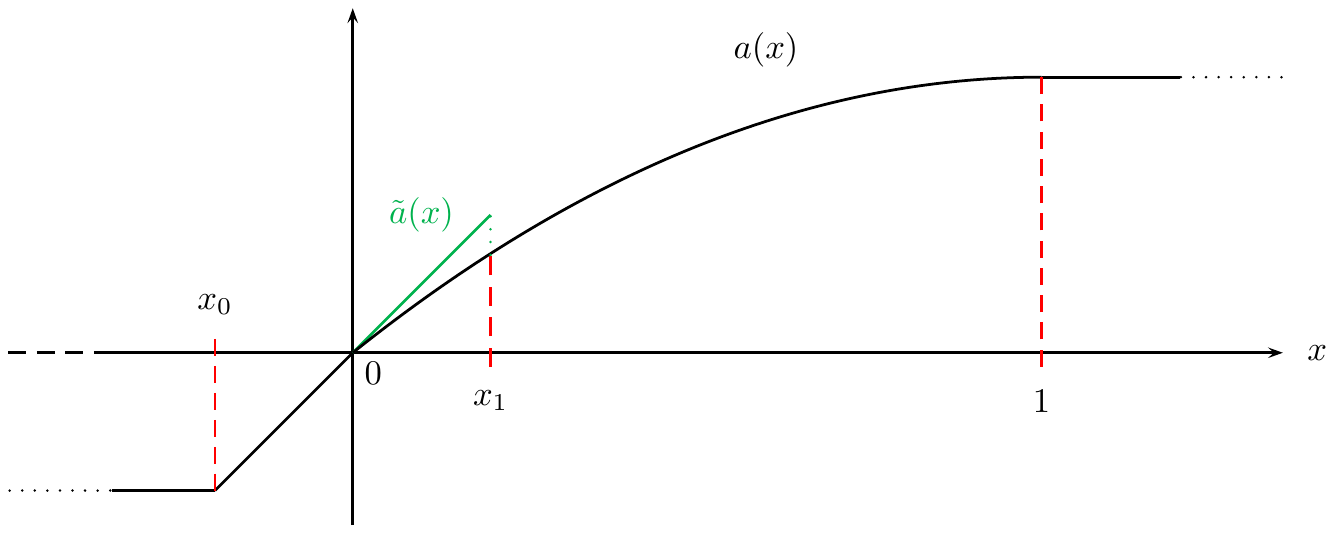}
	\caption{The coefficient function $a(x)$ is linear ($a(x)=x$) on $[x_0,0]$ and quadratic ($a(x)=x-\frac{x^2}{2}$) on $[0,1]$  with a turning point at $x=0$. The green line $\tilde a(x)$ represents a linear approximation on $[0,x_1]$ that is tangential at the turning point.}
	\label{fig:pot_approx2}
\end{figure}

\begin{figure}[!ht]
	\includegraphics[scale=0.75]{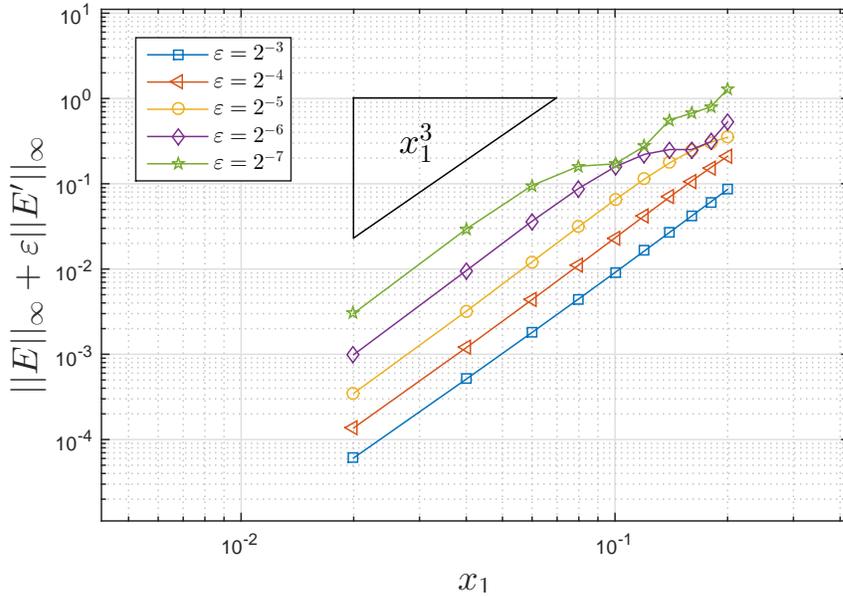}
	\caption{Error of the scattering solution due to the linear approximation of $a(x)$ by $\tilde a(x)$, as shown in Fig.\ \ref{fig:pot_approx2}. The plot shows the error in dependence of $\eps$ and the cut-off point $x_1$; apparently it is of the order $x_1^3/\eps^{1.5}$. The function $E(x)$ is the difference of the two solutions.}
	\label{fig:lin_approx_TP}
\end{figure}

Next we illustrate the consequence of a linear approximation of the potential close to the turning point: We chose a piecewise quadratic coefficient function $a(x)$, see Fig.\ \ref{fig:pot_approx2}, along with its linear approximation $\tilde a$ on $[0,x_1]$ with some (fixed) $x_1\in(0,1)$. 

Since the solutions of the scattering problem for both potentials 
are analytically known in terms of Airy functions and parabolic cylinder functions, the error plotted in Fig.\ \ref{fig:lin_approx_TP} is \emph{not\/} due to any numerical method. The non-uniform error (in $\eps$) stems only from the modified coefficient function, and hence modified phase of the solution: 
For $\eps$ small, the approximate solution is entirely out of phase close to $x=x_1$. 
On the one hand the error in Fig. \ref{fig:lin_approx_TP} is of order $\calO(x_1^3)$ for fixed $\eps$, which stems from the linear approximation of the potential. On the other hand, for fixed $x_1$ the error grows approximately like $\eps^{-3/2}$. 

\medskip

As a conclusion of this feasibility study and as an outlook for a follow-up work there are two options to proceed. Since a linear approximation of $a(x)$ leads to $\eps$--uniform errors for transition points $x_1=o(\eps^{2/5})$ (see \cite{Mil06}), a first option would be to extend the WKB-marching method to $\eps$--dependent intervals of the form $[\eps^\alpha,1]$, with some $0<\alpha<1$. Since the WKB-method yields an $\calO(\eps^3)$--error (for analytically integrable phase functions, see Proposition \ref{prop:wkbmm2}) on $\eps$--independent intervals, an extension to some $[\eps^\alpha,1]$ should yield $\eps$--uniform errors -- hopefully with $\alpha>\frac25$, to have an overlap between both regimes.

A second option is to use Langer functions \cite{Lan31} instead of Airy functions on a fixed interval $[0,x_1]$, coupled to the WKB-marching method. Langer functions are essentially Airy functions with a (generally) non-linear transformation of the argument and a (linear) modification of the amplitude, cf.\ \emph{modified Airy functions\/} in \cite{SBW08}.
This procedure is in strong contrast to the approach by linear approximation of the potential: Both approaches correspond to a transformation to a perturbed Airy equation. The error encountered by using the solution to the reduced problem is non-uniform in $\eps$ for the approach of linear approximation of the potential, but uniform on a fixed interval $[0,x_1]$ for the transformation proposed by Langer, cf.\ \cite[\S 7.2.5]{Mil06}. 
The concept of Langer functions extends also to higher order turning points.


\appendix
\section{\!\!\!}\label{sec:app}
In the first two subsections we shall review known but technical facts that are needed within this article.


\subsection{Asymptotic expansions}\label{app:asy_exp}
First we clarify the notions of \emph{asymptotic expansions\/} and \emph{asymptotic representations\/}, cf.\ \cite{NHM10,Erd56}. 
\begin{definition}[Asymptotic sequences]\label{def_asy_seq}
Let $D$ be a set and $\{\chi_n(x)\}_{n\in \N}$ be a sequence of functions on $D$, such that for $x_0\in D$
\begin{equation}\label{eqn:asy_seq}
\chi_{n+1}(x) =  o(\chi_n(x)), \qquad x\to x_0, \qquad \forall n \in\N.
\end{equation}
Then $\{\chi_n(x)\}_{n\in \N}$ is called an \emph{asymptotic sequence\/} as $x\to x_0$.\footnote{The `little-O' notation $f = o(g)$ as $x\to x_0$ is equivalent to $\frac{f(x)}{g(x)}\to 0$ as $x\to x_0$.}
\end{definition}

\begin{definition}[Asymptotic expansion]\label{def_asy_expan}
Let $\{\chi_n(x)\}_{n\in \N}$ be an asymptotic sequence on $D$ and $x_0\in D$. The (formal) series $\sum_{n=0}^\infty c_n \chi_n(x)$ is called \emph{asymptotic expansion\/} for a function $f(x)$ for $x\to x_0$, shortly denoted as
\begin{equation}\label{eqn_asy_expan}
f(x) \sim \sum_{n=0}^\infty c_n \chi_n(x), \qquad x\to x_0\,,
\end{equation}
if for all $N\in\N\cup\{0\}$,
\begin{equation}\label{eqn_asy_rep}
f(x) = \sum_{n=0}^{N} c_n \chi_n(x) + \calO(\chi_{N+1}(x)), \qquad x\to x_0\,.
\end{equation}
The finite sum on the right hand side is called an \emph{asymptotic representation\/} for $f(x)$ of order $N$. An asymptotic expansion may converge or diverge as $N\to\infty$. 
\end{definition}

For real valued negative arguments, the asymptotic expansions for the Airy function and its derivative are
\begin{equation}\label{eqn:full_asy_airy}
\begin{aligned}
\Ai(-z) &\sim \frac{1}{\sqrt{\pi}z^\frac{1}{4}}\left(\cos\left(\xi-\tfrac{\pi}{4}\right)\sum_{k=0}^\infty (-1)^k\frac{u_{2k}}{\xi^{2k}} + \sin\left(\xi-\tfrac{\pi}{4}\right)\sum_{k=0}^\infty(-1)^k\frac{u_{2k+1}}{\xi^{2k+1}}\right)\,,\\[.5em]
{\Ai}^\prime(-z) &\sim \frac{z^\frac{1}{4}}{\sqrt{\pi}}\left(\sin\left(\xi-\tfrac{\pi}{4}\right)\sum_{k=0}^\infty (-1)^k\frac{v_{2k}}{\xi^{2k}} - \cos\left(\xi-\tfrac{\pi}{4}\right)\sum_{k=0}^\infty(-1)^k\frac{v_{2k+1}}{\xi^{2k+1}}\right)\,,
\end{aligned}
\end{equation}
as $z\to\infty$ where $\xi := \frac{2}{3}z^{3/2}$, cf.\ \cite[\S 9.7(ii)]{NHM10}. Note that these are given in terms of two asymptotic expansions. They are to be interpreted separately in the sense of Definition \ref{eqn_asy_expan}.
The constant coefficients $u_k$ and $v_k$ are defined in \cite[\S 9.7(i)]{NHM10}; here we only need $u_0 = v_0 = 1$, $u_1 = \frac{5}{72}$ and $v_1 = -\frac{7}{72}$.


\subsection{Asymptotic expansions for the PCF including one turning point}\label{app:as_exp_pcf}

The asymptotic expansions for parabolic cylinder functions, i.e.\ solutions to the equation in \eqref{bvp:mod3} with quadratic potential, are given in the literature for a specific form of the differential equation. We will first transform the equation in the following manner:
\begin{align*}
x &\mapsto t = 1 + \frac{2k_1}{k_2}\,x\;,\\
\eps &\mapsto \mu = \frac{k_2}{2\sqrt{\eps}(-k_1^3)^{1/4}}\;,
\end{align*}
which yields
\begin{equation}\label{eqn:t_mu}
\psi_{tt} + \mu^4(1-t^2)\psi =0\;,
\end{equation}
with turning points at $t=\pm 1$. In this form the turning point at $x=0$ corresponds to $t=1$ and the second turning point (originally at $x = -\frac{k_2}{k_1}$) corresponds to $t=-1$. Clearly we have $z(x) = \mu t \sqrt{2}$ and $\nu = -\frac{1}{2}\mu^2$, for $z(x)$ and $\nu$ defined in \eqref{eqn:nu_zx}. Then the solution to the IVP \eqref{ivp:mod3} and its derivative are
\begin{align}\nonumber
\hat\psi(x) &= c_1\,U\left(-\tfrac{1}{2}\mu^2,\mu t\sqrt{2}\right)\,,\\\nonumber
\eps\hat{\psi}^\prime(x) &= -c_1\sqrt{2\eps}(-k_1)^\frac{1}{4}U^\prime\left(-\tfrac{1}{2}\mu^2,\mu t\sqrt{2}\right)\,.
\end{align}

The following asymptotic expansions for the parabolic cylinder function (taken from \cite[\S 12.10(vii)]{NHM10}) hold uniformly in $x(t)\in[0,-\frac{k_2}{k_1}-\delta]$ with $\delta>0$, as $\mu\to\infty$ (or equivalently $\eps\to 0$):
\begin{equation}\label{eqn:full_asy_pcf}
\begin{aligned}
U\left(-\tfrac{1}{2}\mu^2,\mu t\sqrt{2}\right) &\sim 2\,\sqrt{\pi}\mu^\frac{1}{3}g(\mu)\varphi(t)\left(\Ai(\mu^\frac{4}{3}\zeta)\sum_{s=0}^\infty \tfrac{A_s(\zeta)}{\mu^{4s}} + \tfrac{{\Ai}^\prime(\mu^\frac{4}{3}\zeta)}{\mu^\frac{8}{3}}\sum_{s=0}^\infty\tfrac{B_s(\zeta)}{\mu^{4s}}\right)\,,\\[.5em]
U^\prime\left(-\tfrac{1}{2}\mu^2,\mu t\sqrt{2}\right) &\sim \sqrt{2\pi}\mu^\frac{2}{3} \frac{g(\mu)}{\varphi(t)}\left(\tfrac{\Ai(\mu^\frac{4}{3}\zeta)}{\mu^\frac{4}{3}}\sum_{s=0}^\infty \tfrac{C_s(\zeta)}{\mu^{4s}} + {\Ai}^\prime(\mu^\frac{4}{3}\zeta)\sum_{s=0}^\infty\tfrac{D_s(\zeta)}{\mu^{4s}}\right)\,,
\end{aligned}
\end{equation}
where $\zeta = \zeta(t) := -\left(\frac{3}{4}\arccos(t)-\frac{3t}{4}\sqrt{1-t^2}\right)^\frac{2}{3}\leq 0$, which is real valued on the $t$--interval $[-1+\tilde{\delta},1]$ with some $\tilde{\delta}>0$. The function $g(\mu)$ has the following asymptotic expansion w.r.t.\ $\mu\to\infty$:
\begin{equation}\label{eqn:gmu}
g(\mu)\sim h(\mu)\left(1+\frac{1}{2}\sum_{s=1}^{\infty}\frac{\gamma_{s}}{(%
\frac{1}{2}\mu^{2})^{s}}\right)\,,
\end{equation}
where 
\begin{equation}\label{eqn:hmu}
h(\mu):=2^{-\frac{1}{4}\mu^{2}-\frac{1}{4}}e^{-\frac{1}{4}\mu^{2}}\mu^{\frac{1}
{2}\mu^{2}-\frac{1}{2}}\,,
\end{equation}
and the constant coefficients $\gamma_s$ are defined as in \cite[12.10.16]{NHM10} but not further used here. Moreover $\varphi(t)$ is defined as
\begin{equation}\label{eqn:phi}
\varphi(t):=\left(\frac{-\zeta(t)}{1-t^{2}}\right)^{\frac{1}{4}}\;.
\end{equation}
The explicit definitions of $A_s(\zeta)$, $B_s(\zeta)$, $C_s(\zeta)$ and $D_s(\zeta)$ are not needed here and can be found at \cite[12.10.42(44)]{NHM10}. We just note that they are independent of $\eps$ and that 
\begin{equation}\label{eqn:A0D0}
A_0(\zeta)=1\,;\quad D_0(\zeta)=1\,. 
\end{equation}

These asymptotic expansions are in terms of the Airy function and their validity region includes one ($x=0 \Leftrightarrow t=1$) of the two turning points for the quadratic coefficient function. On top of that they hold uniformly for all \mbox{$t\in[-1+\tilde{\delta},1]$}, i.e.\ bounded away from the second turning point at $t=-1$ by some arbitrarily small $\tilde{\delta}>0$. Within the region $[-1+\tilde{\delta},1]$ the coefficient function $1-t^2$ of the transformed equation is non-negative. In order to keep the notation clear, we shall mostly stick to the notation from the literature, i.e.\ using $\mu$ and $\zeta(t)$. 

We also need some properties for the function $\varphi(t)$:
\begin{lemma}\label{lem:phi}
For $t\in[-1+\tilde{\delta},1]$ with any $\tilde{\delta}\in(0,2]$, the function 
\begin{equation*}
\varphi(t) := \left(\frac{-\zeta(t)}{1-t^2}\right)^\frac{1}{4}\,,
\end{equation*}
where $\zeta(t)=-\left(\frac{3}{4}\arccos(t)-\frac{3t}{4}\sqrt{1-t^2}\right)^\frac{2}{3}$, is well-defined, real valued, positive and bounded. In particular it satisfies $\varphi(t)\geq 2^{-\frac{1}{6}}>0$. 
\end{lemma}
\begin{proof}
For $t\in[-1+\tilde{\delta},1)$ it is clearly well-defined and real valued as the numerator is $-\zeta(t)\geq 0$. Moreover $\varphi$ is monotonically decreasing with $\lim_{t\to 1^-}\varphi(t) = 2^{-\frac{1}{6}}$. 
\end{proof}


\subsection{Proofs from Section \ref{subsec:asy_tp}} \label{app:prfs}
We start with the proof of Proposition \ref{prop:asy_pcf} for the asymptotic representations of the parabolic cylinder functions.
\begin{proof}[Proof of Proposition \ref{prop:asy_pcf}]\label{prf:asy_pcf}
The claim is stated for some $x\in(0,-\frac{k_2}{k_1})$ which is equivalent to $t\in(-1,1)$.
Consider the asymptotic expansions \eqref{eqn:full_asy_pcf} for the parabolic cylinder function $U(\nu,z(x))=U(-\frac{1}{2}\mu^2,\mu t\sqrt{2})$. The infinite sums, e.g., $\sum_{s=0}^\infty\tfrac{A_s(\zeta)}{\mu^{4s}}$, are asymptotic series with respect to the asymptotic sequence $\chi_s = \eps^{2s}$. 

With $K:=-\tfrac{k_2^{4/3}}{2^{4/3}\,k_1}>0$, the argument of the Airy function in \eqref{eqn:full_asy_pcf} is $\tfrac{K}{\eps^{2/3}}\zeta(t)=\mu^\frac{4}{3}\zeta(t)<0$ for $t\in(-1,1)$. Also, the asymptotic expansions from \eqref{eqn:full_asy_airy} for the Airy function and its derivative can be written as follows for $\eps\to 0$:
\begin{equation}
\begin{aligned}\label{eqn:asy_exp_airy}
\Ai\left(\mu^\frac{4}{3}\zeta(t)\right) &\sim \tfrac{\eps^\frac{1}{6}}{\sqrt{\pi}(-K\zeta(t))^\frac{1}{4}}\left(A_\eps(t)\sum_{k=0}^{\infty}a_k(t)\eps^{2k} + \eps\,B_\eps(t)\sum_{k=0}^{\infty}b_k(t)\eps^{2k}\right)\,,\\[0.5em]
{\Ai}^\prime\left(\mu^\frac{4}{3}\zeta(t)\right) &\sim \tfrac{(-K\zeta(t))^\frac{1}{4}}{\sqrt{\pi}} \eps^{-\frac{1}{6}}\left(B_\eps(t)\sum_{k=0}^\infty c_k(t)\eps^{2k} + \eps\,A_\eps(t)\sum_{k=0}^\infty d_k(t)\eps^{2k}\right)\,,
\end{aligned}
\end{equation}
where $A_\eps(t) := \cos(\eta(t)-\tfrac{\pi}{4})$, $B_\eps(t):= \sin(\eta(t)-\tfrac{\pi}{4})$ with $\eta(t):=\frac{2(-K\zeta(t))^\frac{3}{2}}{3\,\eps}$ and the first coefficients resp.\ read
\begin{equation}\label{eqn:pcf_coef}
a_0(t) = 1 \,,\quad b_0(t) = \tfrac{5}{48(-K\zeta(t))^{3/2}}\,,\quad c_0(t) = 1\,,\quad d_0(t) = \tfrac{7}{48(-K\zeta(t))^{3/2}}\,.
\end{equation}
Any further coefficients  will not be needed in the proceeding proofs.

While linear combinations of asymptotic expansions are again an asymptotic expansion, it is not generally guaranteed that multiplication will again result in an asymptotic expansion. But if two asymptotic expansions are  power series (so called Poincar\'e-type expansions), their product is again an asymptotic expansion (see \cite[Ch.1, \S 8.1(ii)]{Ol74}). 

The asymptotic expansions we are interested in are power series with respect to the asymptotic sequence $\eps^{2k}$. For some 
$\displaystyle f_\eps(t) \sim \sum_{k=0}^\infty a_k(t) \eps^{2k}$ and $\displaystyle g_\eps(t) \sim \sum_{s=0}^\infty e_s(t) \eps^{2s}$ it holds that
\begin{equation*}
(f_\eps\cdot g_\eps)(t)\sim \sum_{l=0}^\infty c_l(t)\eps^{2l}\;, \quad c_l(t)=\sum_{l=k+s}a_k(t)e_s(t) \,,
\end{equation*}
and thus,
\begin{equation}\label{eqn:asy_rep_prod}
(f_\eps\cdot g_\eps)(t) = a_0(t)e_0(t) + \calO(\eps^2)\,.
\end{equation}

We shall now apply this to the product 
\begin{equation*}
\Ai(\mu^\frac{4}{3}\zeta)\sum_{s=0}^\infty\frac{A_s(\zeta)}{\mu^{4s}}=\Ai(\mu^\frac{4}{3}\zeta)\sum_{s=0}^\infty \frac{A_s(\zeta)}{K^{3s}}\eps^{2s}\,,
\end{equation*}
from \eqref{eqn:full_asy_pcf}. Using $A_0(\zeta)=1$ from \eqref{eqn:A0D0}, and \eqref{eqn:asy_exp_airy} with $a_0(t)$, $b_0(t)$ from \eqref{eqn:pcf_coef}, gives the asymptotic representation
\begin{equation}\label{eqn:asy_prod_1}
\begin{aligned}
&\Ai(\mu^\frac{4}{3}\zeta)\cdot \sum_{s=0}^\infty\tfrac{A_s(\zeta)}{\mu^{4s}} =\\
&= \tfrac{\eps^\frac{1}{6}}{\sqrt{\pi}(-K\zeta)^\frac{1}{4}}\left[\cos(\eta-\tfrac{\pi}{4}) + \tfrac{5}{48(-K\zeta)^\frac{3}{2}}\sin(\eta-\tfrac{\pi}{4})\eps +\calO(\eps^2)\right]\,.
\end{aligned}
\end{equation}
Here $\zeta(t)$ is an $\eps$--independent function in $t$, but $\eta(t)$ is not only $t$--dependent but also of order $\calO(\eps^{-1})$ as $\eps\to 0$. As $\eta(t)$ appears only within $\sin$ and $\cos$, it does not affect the $\eps$--asymptotic behavior of the representation. 

In the same manner we get 
\begin{equation}\label{eqn:asy_prod_2}
\begin{aligned}
&\tfrac{{\Ai}^\prime(\mu^\frac{4}{3}\zeta)}{\mu^\frac{8}{3}}\cdot \sum_{s=0}^\infty\tfrac{B_s(\zeta)}{\mu^{4s}} =\\
&= \tfrac{(-\zeta)^\frac{1}{4}B_0(\zeta)}{\sqrt{\pi}\,K^\frac{7}{4}}\eps^\frac{7}{6}\left[\sin(\eta-\tfrac{\pi}{4}) + \tfrac{7}{48(-K\zeta)^\frac{3}{2}}\cos(\eta-\tfrac{\pi}{4})\eps +\calO(\eps^2)\right]\,.
\end{aligned}
\end{equation}
For the products in the asymptotic expansion of $U^\prime$ in \eqref{eqn:full_asy_pcf} we obtain similarly 
\begin{equation}\label{eqn:asy_prod_3}
\begin{aligned}
&\tfrac{\Ai(\mu^\frac{4}{3}\zeta)}{\mu^\frac{4}{3}}\cdot \sum_{s=0}^\infty\tfrac{C_s(\zeta)}{\mu^{4s}} =\\
&= \tfrac{C_0(\zeta)}{\sqrt{\pi}\,K^\frac{5}{4}(-\zeta)^\frac{1}{4}}\eps^\frac{5}{6}\left[\cos(\eta-\tfrac{\pi}{4}) + \tfrac{5}{48(-K\zeta)^\frac{3}{2}}\sin(\eta-\tfrac{\pi}{4})\eps +\calO(\eps^2)\right]\,;
\end{aligned}
\end{equation}
and
\begin{equation}\label{eqn:asy_prod_4}
\begin{aligned}
&{\Ai}^\prime(\mu^\frac{4}{3}\zeta)\cdot \sum_{s=0}^\infty\tfrac{D_s(\zeta)}{\mu^{4s}} =\\
&= \tfrac{(-K\zeta)^\frac{1}{4}}{\sqrt{\pi}}\eps^{-\frac{1}{6}}\left[\sin(\eta-\tfrac{\pi}{4}) + \tfrac{7}{48(-K\zeta)^\frac{3}{2}}\cos(\eta-\tfrac{\pi}{4})\eps +\calO(\eps^2)\right]\,,
\end{aligned}
\end{equation}
where $D_0(\zeta)=1$.

Putting the pieces together, using the formulas \eqref{eqn:full_asy_pcf} with the asymptotic representations \eqref{eqn:asy_prod_1} -- \eqref{eqn:asy_prod_4} yields
\begin{equation*}
\begin{aligned}
U(\nu,z(x))
&= 2\sqrt{\pi}K^\frac{1}{4}g(\mu)\varphi(t)\eps^{-\frac{1}{6}}\left[\tfrac{\eps^\frac{1}{6}}{\sqrt{\pi}(-K\zeta)^\frac{1}{4}}\cos(\eta-\tfrac{\pi}{4}) +\calO(\eps^\frac{7}{6})\right]\,,\\
& = g(\mu)\left[2\left(\tfrac{1}{1-t^2}\right)^\frac{1}{4}\cos\left(\eta - \tfrac{\pi}{4}\right) + \calO(\eps)\right]\,,
\end{aligned}
\end{equation*}
as well as
\begin{equation*}
\begin{aligned}
\eps\tfrac{d}{dx}U(\nu,z(x))
& = 4 K^\frac{5}{4}\sqrt{\pi}\tfrac{k_1}{k_2}\tfrac{g(\mu)}{\varphi(t)}\eps^\frac{1}{6}\left[ \tfrac{(-K\zeta)^\frac{1}{4}}{\sqrt{\pi}\eps^\frac{1}{6}}\sin\left(\eta - \tfrac{\pi}{4}\right) + \calO(\eps^\frac{5}{6})\right]\,,\\
& = -g(\mu)\left[\tfrac{k_2}{\sqrt{-k_1}}\left(1-t^2\right)^\frac{1}{4}\sin\left(\eta - \tfrac{\pi}{4}\right) + \calO(\eps)\right]\,,
\end{aligned}
\end{equation*}
thus concluding the proof.
\end{proof}

With Proposition \ref{prop:asy_pcf} we can now prove Proposition \ref{prop:bnd_mod3}.

\begin{proof}[Proof of Proposition \ref{prop:bnd_mod3}]\label{prf:sol_order_tp}
For readability of this proof, we omit the index $\eps$ in $\psi_\eps$ (the solution to the BVP \eqref{bvp:mod3}) and $\hat{\psi}_\eps$ (the solution to the IVP \eqref{ivp:tp_mod3}). The proof of Proposition \ref{prop:bnd_mod2} relies on the fact that the scaled Airy function solution on $[0,x_1]$ only depends on the variable $-x\eps^{-2/3}$ (see Remark \ref{rem:airy}). Since this is not true for the PCF solution, the strategy of this proof will deviate from Proposition \ref{prop:bnd_mod2}, and we shall also need the asymptotic expansion of $U(\nu,z(x))$. Still, we make a case distinction similar to the proof of Proposition \ref{prop:bnd_mod2}:

\noindent{\it\underline{Region $x_1\leq x\leq 1$:}}\\
In this region the function $a(x)$ is $\mathcal{C}^2$ and satisfies $a(x)\geq\tau_1>0$. Hence, Lemma \ref{lem:unif_bnd_W_mod3} yields $\|\hat{W}(x)\| = \Theta_\eps(1)$ on $[x_1,1]$ for the vector valued solution $\hat{W}$ to the IVP \eqref{ivp:tp_mod3_W}.
Since the norms of $\hat{W}(x)$ and $(\hat{\psi}(x),\eps\hat{\psi}^\prime(x))^\top$ are ($\eps$--uniformly) equivalent, we get 
\begin{equation*}
\left\|\mtrx{c}{\hat{\psi}(x)\\ \eps\hat{\psi}^\prime(x)}\right\| = \Theta_\eps(1)\,,\quad x_1\leq x\leq 1\,.
\end{equation*}
This yields the asymptotic behavior of the scaling constant $\alpha$ from \eqref{eqn:alpha_scal}: 
\begin{equation}\label{eqn:asy_alpha3}
\alpha(\hat{\psi}(1),\hat{\psi}^\prime(1)) = \Theta_\eps(1)\,.
\end{equation}
For the vector solution of the BVP \eqref{bvp:mod3} in the region $x_1\leq x\leq1$ this yields
\begin{equation}\label{eqn:unif_bnd_notp3}
\left\|\mtrx{c}{\psi(x)\\ \eps\psi^\prime(x)}\right\| = |\alpha|\left\|\mtrx{c}{\hat{\psi}(x)\\ \eps\hat{\psi}^\prime(x)}\right\| = \Theta_\eps(1)\,,\quad x_1\leq x\leq 1\,.
\end{equation}

\noindent{\it\underline{Region $0\leq x\leq x_1$:}}\\
The (extended) PCF solution to the BVP \eqref{bvp:mod3} on $[0,x_1]$ is
\begin{equation}\label{eqn:solpcf_bvp_mod3}
\psi(x) = \tfrac{\alpha(\hat{\psi}(1),\hat{\psi}^\prime(1))}{h(\mu)}\,U\left(\nu,z(x)\right)\,, \quad x\in[0,x_1]\,,
\end{equation}
where $\hat{\psi}(x)$ solves the IVP \eqref{ivp:tp_mod3}. 

We start with the proof of \underline{\smash{statement \ref{item:nonunifbnd_mod3_psi}}}. For the parabolic cylinder function $U(\nu,z(x))$ we can use the asymptotic expansion \eqref{eqn:full_asy_pcf} to get the following asymptotic representation (when using only the first term with $A_0(\zeta)=1$):
\begin{equation}\label{eqn:asy_rep_pcf_tp}
U(\nu,z(x)) = g(\mu)\left(2\sqrt{\pi}\varphi(t)\mu^\frac{1}{3}\Ai(\mu^\frac{4}{3}\zeta(t)) + \calO(\eps)\right)\,,
\end{equation}
with  $\zeta(t) = -\left(\frac{3}{4}\arccos(t)-\frac{3t}{4}\sqrt{1-t^2}\right)^\frac{2}{3}$ and $\varphi(t)=\left(\frac{-\zeta(t)}{1-t^2}\right)^\frac{1}{4}$. This asymptotic representation is uniform in $x\in[0,x_1]\subseteq[0,-\tfrac{k_2}{k_1}-\delta]$ or equivalently $t\in[1+\tfrac{2k_1}{k_2}x_1,1]\subseteq[-1+\tilde{\delta},1]$ for some $\delta,\tilde{\delta}>0$. The function $\varphi(t)$ is well-defined on the interval of interest and takes its minimum value $\varphi(1) =2^{-\frac{1}{6}}$ at the turning point $t=1$ (and $x=0$), cf.\ Lemma \ref{lem:phi}. Using 
\begin{equation*}
\mu^\frac{1}{3} = \frac{k_2^\frac{1}{3}}{2^\frac{1}{3}(-k_1)^\frac{1}{4}\eps^\frac{1}{6}}\,,
\end{equation*}
\eqref{eqn:solpcf_bvp_mod3} and \eqref{eqn:asy_rep_pcf_tp} yields 
\begin{equation*}
|\psi(x)| =|\alpha||f(\mu)|\tfrac{2^\frac{2}{3}\sqrt{\pi}\,k_2^\frac{1}{3}}{(-k_1)^\frac{1}{4}} \left|\varphi(t)\Ai(\mu^\frac{4}{3}\zeta(t))\eps^{-\frac{1}{6}} + \calO(\eps)\right|\,,
\end{equation*}
where $\alpha:=\alpha(\hat{\psi}(1),\hat{\psi}^\prime(1))=\Theta_\eps(1)$ as seen in \eqref{eqn:asy_alpha3}, and $f(\mu):=\tfrac{g(\mu)}{h(\mu)}=\Theta_\eps(1)$.
The asymptotic representation \eqref{eqn:asy_rep_pcf_tp} is uniform in $x\in[0,x_1]$, and thus,
\begin{equation*}
\max_{x\in[0,x_1]}|\psi(x)| =|\alpha||f(\mu)|\tfrac{2^\frac{2}{3}\sqrt{\pi}\,k_2^\frac{1}{3}}{(-k_1)^\frac{1}{4}} \max_{t\in[1+\frac{2k_1}{k_2}x_1,1]}\left|\varphi(t)\Ai(\mu^\frac{4}{3}\zeta(t))\right|\eps^{-\frac{1}{6}} + \calO(\eps)\,.
\end{equation*}
The maximum of the Airy function on $\R$ is attained at the value $y_{max}^{\Ai}\approx -1.01879$. The argument $\mu^\frac{4}{3}\zeta(t)$ takes this value on $[1+\tfrac{2k_1}{k_2}x_1,1]$ 
for $\eps$ sufficiently small, as  $\mu^\frac{4}{3}\zeta(t)$ is a negative (continuous and monotonously increasing) function and zero only at the turning point $t=1$ (and $x=0$), and $\mu^\frac{4}{3}\zeta(t)=\calO(\eps^{-\frac{2}{3}})$. Therefore $M := \max_{z\in\R}|\Ai(z)|= \max_{t\in[1+\frac{2k_1}{k_2}x_1,1]}\left|\Ai(\mu^\frac{4}{3}\zeta(t))\right|$ for $\eps$ sufficiently small. Due to Lemma \ref{lem:phi} the function $\varphi\big|_{[1+\frac{2k_1}{k_2}x_1,1]}$ is real valued, positive and bounded above and below ($\geq 2^{-\frac{1}{6}}$). Hence,
\begin{equation}\label{eqn:max_airy_phi}
\tilde{M}:=\max_{t\in[1+\frac{2k_1}{k_2}x_1,1]}\left|\varphi(t)\Ai(\mu^\frac{4}{3}\zeta(t))\right|=\Theta_\eps(1)\,.
\end{equation}
This implies
\begin{equation*}
\max_{x\in[0,x_1]}|\psi(x)| =C \tilde{M}\eps^{-\frac{1}{6}}+ \calO(\eps)\,,
\end{equation*}
where the constant $C=\Theta_\eps(1)$, and thus, $\|\psi\|_{L^\infty(0,x_1)}=\Theta(\eps^{-\frac{1}{6}})$. Together with \eqref{eqn:unif_bnd_notp3} this yields 
\begin{equation*}
\|\psi\|_{L^\infty(0,1)} = \Theta(\eps^{-\frac{1}{6}})\,,\qquad \eps\to 0\,.
\end{equation*}
\medskip

For the \underline{\smash{statement \ref{item:unifbnd_mod3_psiPri}}} we shall use again the uniform asymptotic expansion \eqref{eqn:full_asy_pcf} (when using only the second term with $D_0(\zeta)=1$) -- this time for 
\begin{equation*}
\eps\hat{\psi}^\prime(x) = -\tfrac{\sqrt{2\eps}(-k_1)^\frac{1}{4}}{h(\mu)}\,U^\prime(\nu,z(x))\,,
\end{equation*}
where the hat indicates the solution of the IVP \eqref{ivp:tp_mod3}. After scaling with the constant $\alpha$, we get for the solution to the BVP \eqref{bvp:mod3}:
\begin{align*}
\eps\psi^\prime(x) &= -\alpha\tfrac{\sqrt{2\eps}(-k_1)^\frac{1}{4}}{h(\mu)}\,U^\prime(\nu,z(x))\\
&= -\alpha\,f(\mu)\left(\tfrac{2^\frac{1}{3}\sqrt{\pi}\,k_2^\frac{2}{3}}{(-k_1)^\frac{1}{4}}\,\frac{\eps^\frac{1}{6}{\Ai}^\prime(\mu^\frac{4}{3}\zeta(t))}{\varphi(t)} + \calO(\eps^\frac{5}{6})\right)\,.
\end{align*}
Since this asymptotic representation is uniform in $x\in[0,x_1]$, we can write
\begin{equation*}
\max_{x\in[0,x_1]}|\eps\psi^\prime(x)| = |\alpha||f(\mu)|\tfrac{2^\frac{1}{3}\sqrt{\pi}\,k_2^\frac{2}{3}}{(-k_1)^\frac{1}{4}}\max_{t\in[1+\frac{2k_1}{k_2}x_1,1]}\left|\frac{\eps^\frac{1}{6}{\Ai}^\prime(\mu^\frac{4}{3}\zeta(t))}{\varphi(t)}\right| + \calO(\eps^\frac{5}{6})\,.
\end{equation*}
The argument of ${\Ai}^\prime$, i.e.\ $\mu^\frac{4}{3}\zeta(t)= \frac{K}{\eps^{2/3}}\zeta(t)$, is negative (except at the turning point $t=1$). The function $-K\zeta(t)$ maps the $t$--interval $[1+\frac{2k_1}{k_2}x_1,1]$ onto the interval $[0,\hat{x}]$ for some $\hat{x}>0$. Therefore the proof that $|\eps^\frac{1}{6}\,{\Ai}^\prime(\mu^\frac{4}{3}\zeta(t))|$ is $\eps$--uniformly bounded on $[1+\frac{2k_1}{k_2}x_1,1]$ is analog to Step 5 of the proof Proposition \ref{prop:bnd_mod2}.
Using $\varphi(t)\geq 2^{-\frac{1}{6}} > 0$ yields
\begin{equation}\label{eqn:max_ariyPri_phi}
\max_{t\in[1+\frac{2k_1}{k_2}x_1,1]}\left|\frac{\eps^\frac{1}{6}{\Ai}^\prime(\mu^\frac{4}{3}\zeta(t))}{\varphi(t)}\right|\leq \tilde{C}\,,
\end{equation}
for some $\eps$--independent $\tilde{C}>0$. Hence, $\eps\|\psi^\prime\|_{L^\infty(0,x_1)} = \calO(1)$ as $\eps\to 0$. Together with \eqref{eqn:unif_bnd_notp3} this yields 
\begin{equation*}
\eps\|\psi^\prime\|_{L^\infty(0,1)} = \calO_\eps(1)\,.
\end{equation*}
\end{proof}

\bigskip
\bigskip


\textbf{Acknowledgement:}

The authors were partially supported by the FWF-funded doctoral school W1245 and the FWF-project I3538--N32. Moreover, the first author is grateful to Houde Han for stimulating discussions on highly oscillatory problems. 



\end{document}